\documentclass[a4paper,11pt]{article}

\usepackage{fullpage}
\usepackage{graphicx,xcolor}
\usepackage{amsmath,amsfonts,amsthm,amssymb}
\usepackage{float}

\usepackage{pgfplots}
\pgfplotsset{compat=newest}
\usetikzlibrary{plotmarks}
\usetikzlibrary{arrows.meta}
\usepgfplotslibrary{patchplots}
\usepgfplotslibrary{fillbetween}
\usepackage{grffile}
\usepackage{hhline}

\newcommand{\ee}{\mathrm{e}}
\newcommand{\RR}{\mathbb{R}}
\newcommand{\CC}{\mathbb{C}}
\newcommand{\ZZ}{\mathbb{Z}}
\newcommand{\bb}{\boldsymbol}
\newcommand{\bx}{\bb x}

\newtheorem{remark}{Remark}[section]
\newtheorem{theorem}{Theorem}[section]

\newcommand{\keywords}[1]{\noindent \textbf{\textit{Keywords:}} #1}
\newcommand{\msccodes}[1]{\noindent \textbf{\textit{MSC codes (2020):}} #1}
\usepackage{hyperref}

\title{Accelerating exponential integrators to efficiently solve\\
  {\color{black}semilinear} advection-diffusion-reaction equations}
\author{Marco Caliari\footnote{\noindent Department of Computer Science, University of
        Verona, Verona, Italy.\newline
        \textit{E-mail address}: \texttt{marco.caliari@univr.it}}
        \and Fabio Cassini\footnote{Department of Computer Science, University of
        Verona, Verona, Italy.\newline
        \textit{E-mail address}: \texttt{fabio.cassini@univr.it}}
        \and Lukas Einkemmer\footnote{Department of Mathematics, University of
        Innsbruck, Innsbruck, Austria.\newline
        \textit{E-mail address}: \texttt{lukas.einkemmer@uibk.ac.at}}
        \and Alexander Ostermann\footnote{Department of Mathematics, University of
        Innsbruck, Innsbruck, Austria.\newline
        \textit{E-mail address}: \texttt{alexander.ostermann@uibk.ac.at}}
}

\begin{document}
  \maketitle
\begin{abstract}
In this paper we consider an approach to improve the performance of exponential
{\color{black}Runge--Kutta integrators and Lawson schemes} in cases where the 
solution of a related, but
usually much simpler, problem can be computed efficiently. While for implicit
methods such an approach is common (e.g.~by
using preconditioners), for exponential integrators this has proven more
challenging. Here we propose to extract a constant coefficient differential
operator from the {\color{black}semilinear} advection-diffusion-reaction equation
for which, {\color{black}in many situations, efficient methods are known
  to compute the required matrix functions.} Both a
linear stability analysis and {\color{black} extensive} numerical experiments 
show that the resulting
schemes can be unconditionally stable. In fact, we find that exponential
integrators {\color{black}of Runge--Kutta type} and Lawson schemes can 
have better stability properties than
similarly constructed implicit-explicit schemes. We also {\color{black} derive two}
new Lawson type integrators that further improve on these stability properties.
The {\color{black} overall} effectiveness of the approach is highlighted by a 
number of {\color{black} performance comparisons on} examples in two and 
three space dimensions.
\end{abstract}


\medskip
\msccodes{65L05, 65M06, 65M20, 65M70}
\section{Introduction}

Solving time dependent Partial Differential Equations (PDEs) efficiently is
important for better understanding many phenomena in science and engineering. In
most cases, in particular for problems which are diffusion dominated,
discretizing such PDEs in space results in a very stiff set of Ordinary
Differential Equations (ODEs) for which explicit integrators have to take
extremely small time steps.

The standard approach to reduce the computational effort for such problems is
to use implicit methods. {\color{black}These kinds} of methods, from a theoretical point of
view, can have very attractive properties. In particular, many implicit methods
are \emph{unconditionally stable}. {\color{black}Since in this paper
we are mostly interested in diffusion dominated problems}, we consider methods that are
unconditionally stable in the sense that they do not have a time step
restriction when applied to the linear problem
\begin{equation*}
  u'(t)=Au(t),
\end{equation*}
where $A$ is a matrix with negative eigenvalues.
That is, the time step size is only dictated by the
prescribed tolerance (i.e., accuracy), but not by stability considerations.
However, implicit methods in general require the solution of a nonlinear set
of equations at each time step. This is usually done either by performing Newton
iteration in combination with some linear solver
or by directly using an IMplicit-EXplicit (IMEX) scheme, which treats
implicitly only the linear part of the system.
Thus, using implicit methods shifts all the difficulty into finding an
efficient way to solve dense linear systems of
moderate size (such as those coming from pseudospectral techniques)
or large and sparse ones
(such as those stemming from finite differences or finite elements spatial
discretizations).
When the arising linear
system is solved by a direct method such as LU decomposition
the computational complexity
of the obtained algorithm is usually worse than that of an explicit
integrator. Otherwise, Krylov subspace based iterative methods are used,
such as the
Conjugate Gradient (CG) method or the Generalized Minimal RESidual method
(GMRES)~\cite{saad2003iterative}. For diffusion dominated problems Krylov
methods already significantly reduce the computational effort needed. However,
for a number of problems preconditioners can be constructed that further
reduce {\color{black}it}. Preconditioners transform the linear system
in such a way that the solution of the resulting linear system requires fewer
iterations. While there are preconditioners that, at least in principle, can
be applied generically (e.g.~ILU or the celebrated algebraic multi-grid
approach~\cite{xu2017algebraic}), for many challenging problems purpose-built
preconditioners have to be constructed. Often such preconditioners use the
solution of a related, but simpler, problem that can be solved efficiently
(see, e.g., physics based preconditioners~\cite{Chacon2002,Reynolds2010} or
block preconditioners~\cite{badia2014block,cyr2016teko}).

Exponentials integrators (see reference~\cite{HO10} for a review) are
another way to
solve in time stiff differential equations
{\color{black}in which} a linear part of the problem is
treated \emph{exactly}, while the nonlinear remainder is treated in an
explicit manner. 
{\color{black} In this work, we focus on exponential integrators of Runge--Kutta
type~\cite{HO05} and on Lawson schemes~\cite{lawson} (i.e., methods that 
just involve matrix exponentials but not related functions)}.
For example, to solve the equation
\[ \partial_t u(t) = A u(t) + g(t, u(t)) \]
it is possible to use the exponential Euler method
\begin{equation*}
u^{n+1}=u^{n}+\tau\varphi_{1}(\tau A)(A u^{n} + g(t_n, u^n))
\end{equation*}
{\color{black} or the Lawson--Euler scheme
\begin{equation*}
u^{n+1}=\ee^{\tau A}(u^{n}+\tau g(t_n, u^n)).
\end{equation*}
}%
Here $A$ is a linear operator/matrix, $\tau$ is the time step size, and
$\varphi_1(z) = (\ee^z - 1)/z$ is an entire function (the singularity at $0$ is
removable). Even for nonlinear problems,
exponential integrators require no Newton iteration, but rather
 the computation of the action of matrix functions related to the
exponential (such as the $\varphi_1$ function above) at each time step.
The main difficulty for exponential integrators is the evaluation of
{\color{black}such} exponential-like matrix
functions. For that purpose Krylov subspace methods~\cite{hochbruck1997krylov},
schemes based on Leja interpolation~\cite{CKOR16,caliari2004}, and Taylor
methods~\cite{al2011computing} are commonly used.
It has been shown in a number of studies that Krylov or Leja based
exponential integrators can be superior to implicit schemes for realistic
problems, see, e.g., references~\cite{clancy2013use,deka2022exponential,einkemmer2017performance,loffeld2013comparative}.
Unfortunately, however, the aforementioned technique of preconditioning
to further improve
performance does not apply to exponential integrators.

Thus, a natural question that arises is whether the solution of a related, but
simpler, problem for which an efficient solution is known can be used in the
framework of exponential integrators in order to improve the performance of
the scheme. The idea that we pursue in this paper is to choose the linear
operator $A$ in such a
way that the corresponding matrix functions can be computed efficiently.
In particular, {\color{black} we are interested in
(diffusion dominated) variable coefficients semilinear advection-diffusion-reaction
equations in the following form}
\begin{equation*}
  \partial_tu(t,\bx)=\nabla\cdot\left(a(\bx)\nabla u(t,\bx)\right)+
  \nabla\cdot\left(\bb b(\bb x)u(t,\bb x)\right)+
  r(t,\bx,u(t,\bx)),
    \quad  t\in[0,T],\quad \boldsymbol{x}\in\Omega\subset\RR^d,
\end{equation*}
equipped with appropriate initial and boundary conditions.  If we use
\begin{equation*}
  Au = \nabla \cdot \left(a(\bx) \nabla u\right)+
  \nabla\cdot\left(\bb b(\bb x)u\right)
\end{equation*}
  then we
have to resort to
general purpose Krylov or Leja based schemes, for instance,
to compute the actions of {\color{black}$\ee^{\tau A}$ and}
$\varphi_1(\tau A)$.
However, for constant diffusion tensor and velocity
field (i.e., if $a$ and $\bb b$ are independent of
$\bx$) we can use a fast Poisson solver, e.g.~Fast Fourier Transform (FFT)
based methods. This scales linearly
(up to a logarithm) in the number of unknowns. The idea is then to consider
the following \emph{equivalent} formulation of the equation
\begin{equation*}
    \partial_tu(t,\bx)=A u(t,\bx)
    +\underbrace{\nabla\cdot(a(\bx)\nabla u(t,\bx))
+  \nabla\cdot\left(\bb b(\bb x)u(t,\bb x)\right)
      -A u(t,\bx)+r(t,\bx,u(t,\bx))}_{g(t,\bx,u(t,\bx))},
\end{equation*}
where now $A$ is a constant coefficient approximation to the original
advection-diffusion operator. The hope is that
the actions of {\color{black}$\ee^{\tau A}$ and} $\varphi_1(\tau A)$ can
stabilize the nonlinear term $g(t,\bx,u(t,\bx))$ of the
exponential integrator. The choice of $A$ is clearly not uniquely determined
and we will consider how to choose it later in the paper. {\color{black} We remark,
  however, that $A$ has to be chosen in such a way that we can efficiently
  compute the corresponding matrix functions.}%

{\color{black}Let us note that a related approach to what we pursue in this paper
leads to split or partitioned exponential
integrators~\cite{NS21,NTST19,RT14}, which has also been
applied to non-diffusive
equations~\cite{BM22}. The idea of employing a constant coefficient operator can
be considered in the context of implicit methods as well. In particular,} an
IMEX scheme can be used. The
approximation operator is treated implicitly (which can be done efficiently
using, e.g., Fourier methods), while the nonlinear term
is treated explicitly.
Such schemes, known in the literature as explicit-implicit-null,
were considered for instance in reference~\cite{WZWS20}.
Clearly,
the class of operators for which efficient linear solvers are known is
not equivalent to the class of operators for which the matrix exponential
(and related matrix functions) can be computed efficiently.
{\color{black}For example, it is quite involved to build a linear solver if a
matrix has $d$-dimensional Kronecker sum form~\cite{CK20}. Instead, 
this structure can be more easily exploited for the matrix exponential
case~\cite{CCEOZ22,CCZ22b}.
Indeed, thanks to the equivalence
\begin{equation*}
  \ee^{A_d\oplus A_{d-1}\oplus\cdots\oplus A_1}\bb v =
  \mathrm{vec}\left(\bb V \times_1 \ee^{A_1} \times_2 \cdots \times_d \ee^{A_d}\right),
  \quad
  \bb V \in \CC^{N_{x_1}\times \cdots \times N_{x_d}}, \quad \bb v = \mathrm{vec}(\bb V),
\end{equation*}
we are able to compute the action of the exponential of a matrix in Kronecker sum form just
in terms of the small sized matrix exponentials $\ee^{A_\mu}$ and
tensor-matrix $\mu$-mode products (here denoted with the symbol $\times_\mu$) using
$d$ calls to level 3 BLAS. This idea leads to the so-called
$\mu$-mode integrator, with dramatic improvements in terms of efficiency
of computation.}
On the other hand, if for a specific operator a good preconditioner is known,
an efficient linear solver can be constructed, while the same is not
necessarily true for computing the corresponding matrix functions.

Let us also mention some related work in the context of computing matrix
functions. In reference~\cite{CS98} the fact that the matrix exponential
can be written
as the solution of a linear differential equation is used to decompose the
problem into a ``preconditioner'' and a correction. An exponential integrator is
then used to approximate the resulting system. The preconditioners chosen are
all purely algebraic in nature (diagonal, block diagonal, etc.) and are not
specifically adapted to the underlying PDE.
In reference~\cite{van2006preconditioning}
a rational Krylov subspace is constructed. That is, the inverse of the
preconditioner is directly built into the Krylov subspace in order to obtain
an approximation space that is well behaved. This approach was later extended
to trigonometric operators \cite{grimm2008rational}.

In this paper we will first perform a linear stability analysis of the
proposed approach. In particular, this will show that exponential integrators
can have better stability properties than their corresponding IMEX
counterparts. We will also propose a class of methods (called stabilized
Lawson schemes) that improve on the stability properties of the classic
Lawson schemes. We investigate the choice of {\color{black}the operator} $A$ and show that this can have a
drastic influence on performance. We note that in reference~\cite{WZWS20},
in the context
of IMEX schemes, the authors employ the
smallest amount of
diffusion in $A$ that is necessary to guarantee stability. We do observe
a similar behavior for some test examples. However, for different two-
and three-dimensional {\color{black}semilinear} advection-diffusion-reaction equations that we investigate
this is not the case. In this context we {\color{black} state our proposal,
also based on strong numerical evidence,} to select $A$ in a good way.
We then perform a number of numerical investigations
with appropriate techniques for evaluating the
matrix functions in order to highlight the efficiency of the proposed approach.
Finally, we show that {\color{black}the investigated accelerated exponential methods} can
outperform their IMEX counterparts in a number of situations.
\section{Linear stability analysis}\label{sec:linstaban}
As mentioned in the introduction, we investigate here how
to determine the approximation operator $A$ {\color{black}for exponential 
Runge--Kutta and Lawson methods} by studying a model equation.
In particular, we consider the constant coefficient heat equation
\begin{equation*}
\partial_tu=\Delta u
\end{equation*}
on the domain $\Omega=(-\pi,\pi)^{d}$ with periodic boundary conditions, and we
equivalently write
\begin{equation}\label{eq:modeleqstab}
\partial_tu=\lambda\Delta u + (1-\lambda)\Delta u,
\end{equation}
with $\lambda\in[0,1]$.
Note that in this case the approximation operator is simply
\begin{equation*}
A=\lambda\Delta.
\end{equation*}
We note that a similar analysis can be performed for more general operators as
long as the eigenvalues lie on the negative real axis.
Then, in order to determine the parameter $\lambda$ we perform a linear
stability analysis of the temporal exponential integrator in Fourier space,
{\color{black} which appears to be novel in the literature (see
reference~\cite[Sec. 4.1]{DASS20} for a similar investigation, in a different
context, for the exponential Euler method only)}.
For the convenience of the reader, all the schemes mentioned and studied in
this section are collected in the appendix (formulated for a generic abstract
semilinear ODE).

Let us first consider the well-known exponential Euler method, which integrates
equation~\eqref{eq:modeleqstab}
in time as follows
\begin{equation}\label{eq:expeulerstab}
u^{n+1}=u^{n}+\tau\varphi_{1}(\tau\lambda\Delta)\Delta u^{n},
\end{equation}
where $u^n$ is the numerical solution at time $t_n$, $\tau$ is the time
step size and $\varphi_1(z)=(\ee^z-1)/z$. Here and throughout the paper we
assume without loss of generality
that $\tau$ is constant. Then we have the following result.
\begin{theorem}\label{thm:expeuler}
The exponential Euler scheme~\eqref{eq:expeulerstab}
is unconditionally stable for $\lambda\geq\lambda^{\textsc{ee}}=1/2$.
\end{theorem}

\begin{proof}
  Let us denote $\bb k = (k_1,\ldots,k_d)\in\ZZ^d$,
  $k^2=\sum_\mu k_\mu^2$, and let
$\hat{u}_{\bb k}^n$ be the $\bb k$th Fourier mode of $u^n$. Then, in Fourier
space we have
\begin{equation*}
\Phi(z, \lambda):=\frac{\hat{u}_{\bb k}^{n+1}}{\hat{u}_{\bb k}^{n}}=
1-\varphi_{1}(-\lambda\tau k^{2})\tau k^{2}=
1-\frac{1}{\lambda}+\frac{\ee^{z\lambda}}{\lambda},
\end{equation*}
where $z = -\tau k^2$.
Thus, we have unconditional stability if the stability function
$\Phi(z, \lambda) $ satisfies
\[ \vert \Phi(z,\lambda) \vert \leq 1 \quad \text{for all} \quad z \in (-\infty,0]. \]
Since $\lambda \in [0,1]$ this implies $\lambda\geq1/2$,
as desired.
\end{proof}
What Theorem~\ref{thm:expeuler} tells us is that we still get an
unconditionally stable scheme for the heat equation even if we only put
half of the Laplacian into $A$.
This is the key observation for the
numerical methods that we propose in this paper.

Let us now turn our attention to the Lawson--Euler scheme
\begin{equation}\label{eq:lawsoneulerstab}
u^{n+1}=\ee^{\tau\lambda\Delta}(u^{n}+\tau(1-\lambda)\Delta u^{n}),
\end{equation}
which can alternatively be seen as a Lie splitting where we approximate
the second subflow by explicit Euler, that is
\begin{equation*}
\ee^{\tau\lambda \Delta}\ee^{\tau(1-\lambda)\Delta}u^{n}
\approx \ee^{\tau\lambda\Delta}(u^{n}+\tau(1-\lambda)\Delta u^{n}) = u^{n+1}.
\end{equation*}
For this method we obtain the following result.
\begin{theorem}\label{thm:lawsoneuler}
The Lawson--Euler scheme~\eqref{eq:lawsoneulerstab}
is unconditionally stable for $\lambda\geq\lambda^{\textsc{le}}=0.218$.
\end{theorem}
\begin{proof}
The stability function satisfies
\begin{equation*}
\Phi(z, \lambda) =
\ee^{\lambda z}(1+(1-\lambda)z).
\end{equation*}
Taking the minimum with respect to $z$ we get
\[ z = \frac{1}{\lambda(\lambda-1)}. \]
Plugging this into the stability function we obtain
\begin{equation*}
  1-\frac{1}{\lambda} +\ee^{\frac{1}{1-\lambda}} \geq 0.
\end{equation*}
The result then simply follows by computing the zero of the left-hand side.
\end{proof}

This shows that there exist methods which are unconditionally stable
for values of $\lambda$ smaller than $\lambda^{\textsc{ee}}=1/2$.
This is of interest because it is often the case that the accuracy of the
method increases as $\lambda$
decreases (see the experiments in sections~\ref{sec:validation}
and~\ref{sec:numexpstab}).
In addition, as mentioned in the introduction,
the advantage of using schemes
that just employ the exponential function is the following:
in certain situations
it can be more efficient to compute the exponential than the $\varphi$ functions
(e.g., for exploiting the Kronecker structure or if a semi-Lagrangian
scheme is used).

A natural question to ask is if it is possible to construct a numerical method
for which the constraint on $\lambda$ is further reduced.
We propose then the following scheme
\begin{equation}\label{eq:stabilizedlawsoneulerstab}
u^{n+1}=u^{n}+\tau \ee^{\tau\lambda\Delta}\Delta u^{n}.
\end{equation}
We call this first order method the \emph{stabilized Lawson--Euler} scheme,
for which we obtain the following result.
\begin{theorem}\label{thm:stabilizedlawsoneuler}
  The stabilized Lawson--Euler scheme~\eqref{eq:stabilizedlawsoneulerstab} is
  unconditionally stable
for $\lambda\geq\lambda^{\textsc{sle}}=1/(2\ee)\approx0.184$.
\end{theorem}
\begin{proof}
In this case, the stability function is given by
\begin{equation*}
\Phi(z, \lambda) = 1 + z \ee^{\lambda z}.
\end{equation*}
The stated result then follows immediately.
\end{proof}

\begin{remark}\label{rem:imex} \rm
A similar analysis can be performed also for other classes of schemes.
For example, in reference~\cite{WZWS20} some IMEX schemes have been analyzed in
a fully discretized context. In particular, for the well known
backward-forward Euler method
\begin{equation*}
(I-\tau\lambda\Delta)u^{n+1}=(I+\tau(1-\lambda)\Delta)u^{n}
\end{equation*}
the authors obtain the bound
$\lambda\geq\lambda^{\textsc{bfe}}=1/2$ for unconditional stability.
Moreover, they propose the second order
method
\begin{equation*}
\begin{aligned}
  \left(I-\frac{\tau}{2}\lambda\Delta \right)U&=\left(I+
   \frac{\tau}{2}(1-\lambda)\Delta \right)u^n, \\
  \left(I-\frac{\tau}{2}\lambda\Delta\right)u^{n+1}&=
   \left(I+\frac{\tau}{2}\lambda\Delta\right)u^n + \tau(1-\lambda)\Delta U
\end{aligned}
\end{equation*}
which has the same stability bound $\lambda\geq\lambda^{\textsc{imex2}}=1/2$.
\end{remark}

Let us now consider some examples of second order exponential integrators.
In particular, we start with the following
exponential Runge--Kutta type scheme
for equation~\eqref{eq:modeleqstab}
\begin{equation}\label{eq:exprk2phi2stab}
\begin{aligned}
U &= u^n + c_2\tau\varphi_1(c_2\tau\lambda\Delta)\Delta u^n, \\
u^{n+1} &= u^n + \tau\varphi_1(\tau\lambda\Delta)\Delta u^n +
\frac{\tau}{c_2}\varphi_2(\tau\lambda\Delta)(1-\lambda)\Delta(U-u^n),
\end{aligned}
\end{equation}
where $0<c_2\leq 1$ is a free parameter and $\varphi_2(z) = (\varphi_1(z)-1)/z$.
For this scheme we have the following result.
\begin{theorem}\label{thm:exprk2phi2}
  The exponential Runge--Kutta scheme~\eqref{eq:exprk2phi2stab} is
  unconditionally stable for $\lambda\geq 1/(1+c_2)$.
\end{theorem}
\begin{proof}
The stability function is given by
\begin{equation*}
\Phi(z, \lambda) = 1 +  z\varphi_1(\lambda z) +
z^2(1-\lambda)\varphi_2(\lambda z)\varphi_1 (c_2 \lambda z).
\end{equation*}
Then, by imposing
$\lvert \Phi(z, \lambda) \rvert \leq 1$ for all $z\in(-\infty,0]$ we have
\begin{equation*}
 \frac{1-\lambda}{c_2\lambda}- 1 \leq 0
\end{equation*}
from which we obtain $\lambda \geq 1/(1+c_2)$ as desired.
\end{proof}
The smallest result, in terms of unconditional stability, is obtained by
setting the free parameter $c_2=1$, which yields $\lambda \geq
\lambda^{\textsc{erk2p2}}=1/2$.

Let us consider now another second order exponential Runge--Kutta
integrator, which requires just the $\varphi_1$ function. For the model
equation~\eqref{eq:modeleqstab} this yields
\begin{equation}\label{eq:exprk2phi1stab}
\begin{aligned}
U &= u^n + c_2\tau\varphi_1(c_2\tau\lambda\Delta)\Delta u^n, \\
u^{n+1} &= u^n + \tau\varphi_1(\tau\lambda\Delta)\Delta u^n +
\frac{\tau}{2c_2}\varphi_1(\tau\lambda\Delta)(1-\lambda)\Delta(U-u^n).
\end{aligned}
\end{equation}
We  obtain the following unconditional stability result.
\begin{theorem}\label{thm:exprk2phi1}
  The exponential Runge--Kutta scheme~\eqref{eq:exprk2phi1stab} is
  unconditionally stable
for $\lambda\geq1/(1+2c_2)$.
\end{theorem}
\begin{proof}
  The stability function is given by
\begin{equation*}
\Phi(z, \lambda) = 1 + z\varphi_1(\lambda z) +
\frac{z^2}{2}(1-\lambda)\varphi_1(\lambda z)\varphi_1 (c_2 \lambda z)
\end{equation*}
for which we obtain $\lambda \geq 1/(1+2c_2)$ as desired.
\end{proof}
The choice $c_2=1$ yields the bound
$\lambda \geq \lambda^{\textsc{erk2p1}}=1/3$,
which is smaller than the bound obtained for the previous
method.

Concerning the second order Lawson type schemes, we consider the
Lawson2a and the Lawson2b integrators, which integrate
equation~\eqref{eq:modeleqstab} as follows
\begin{equation}
    \begin{aligned}\label{eq:lawson2astab}
      U &= \ee^{\frac{\tau}{2}\lambda\Delta}\left(u^n +
      \tfrac{\tau}{2}(1-\lambda)\Delta u^n\right), \\
      u^{n+1} &= \ee^{\tau\lambda\Delta}u^n +
      \tau\ee^{\frac{\tau}{2}\lambda\Delta}(1-\lambda)\Delta U,
    \end{aligned}
  \end{equation}
and
  \begin{equation}\label{eq:lawson2bstab}
    \begin{aligned}
      U &= \ee^{\tau\lambda\Delta}\left(u^n + \tau(1-\lambda)\Delta u^n\right), \\
      u^{n+1} &= \ee^{\tau\lambda\Delta}u^n + \frac{\tau}{2}\ee^{\tau\lambda\Delta}(1-\lambda)\Delta u^n + \frac{\tau}{2}(1-\lambda)\Delta U,
    \end{aligned}
  \end{equation}
respectively. In terms of unconditional stability, we have the following result.
\begin{theorem}\label{thm:lawson2ab}
  The Lawson2a scheme~\eqref{eq:lawson2astab} and the Lawson2b
  scheme~\eqref{eq:lawson2bstab} are unconditionally stable
for $\lambda\geq\lambda^{\textsc{l2a}}=\lambda^{\textsc{l2b}}=0.301$.
\end{theorem}
\begin{proof}
  The stability functions are given by
\begin{equation*}
  \Phi(z, \lambda) =\ee^{\lambda z}\left(
   1+ z(1-\lambda)\left(1+\frac{z(1-\lambda)}{2}\right)\right)
\end{equation*}
and
\begin{equation*}
  \Phi(z, \lambda) =  \ee^{\lambda z}\left(1+\frac{z}{2}(1-\lambda)+\frac{z}{2}(1-\lambda)(1+z(1-\lambda))\right)
\end{equation*}
for Lawson2a and Lawson2b, respectively. In both cases, a numerical calculation
leads to the restriction $\lambda \geq 0.301$ for unconditional stability.
\end{proof}

We now propose a second order variant of the stabilized Lawson--Euler scheme.
Letting $0<\alpha\leq 1$ be a free parameter, we obtain for the model
equation~\eqref{eq:modeleqstab} the following method
\begin{equation}\label{eq:stablawson2stab}
\begin{aligned}
U &= u^n + \alpha\tau\ee^{\alpha\tau\lambda\Delta}\Delta u^n, \\
u^{n+1} &= u^n + \tau\ee^{\frac{\tau}{2}\lambda\Delta}\Delta u^n +
\frac{\tau}{2\alpha}\ee^{\tau\lambda\Delta}(1-\lambda)\Delta(U-u^n).
\end{aligned}
\end{equation}
We call this second order integrator the \textit{stabilized Lawson2} scheme, for
which we have the following result.
\begin{theorem}\label{thm:stablawson2}
  The stabilized Lawson2 scheme~\eqref{eq:stablawson2stab} with
  $\alpha=\alpha^{\textsc{sl2}}=0.327$
is unconditionally stable
for $\lambda\geq\lambda^{\textsc{sl2}}=0.183$.
\end{theorem}
\begin{proof}
The stability function is given by
\begin{equation*}
  \Phi(z, \lambda) = 1 + z\ee^{\frac{\lambda z}{2}} +
  \frac{z^2}{2}(1-\lambda)\ee^{(1+\alpha)\lambda z}.
\end{equation*}
Then, setting $\alpha=\alpha^{\textsc{sl2}}=0.327$, a numerical calculation
leads to the
restriction $\lambda \geq 0.183$ for unconditional stability.
\end{proof}
Note that the value $\alpha^{\textsc{sl2}}=0.327$ is chosen such that the
corresponding parameter $\lambda^{\textsc{sl2}}$ is as small as possible.

In Table~\ref{tab:stab} we have collected, in descending order,
the stability lower bounds of all the methods considered in this section,
together with the labels that we will use throughout the paper.
To summarize, we find that the stabilized Lawson schemes can be
operated with the smallest values of $\lambda$, followed by the Lawson
schemes. Exponential integrators and IMEX schemes require a larger
value of $\lambda$ to be unconditionally stable.
\begin{table}[htb!]
  \centering
  \begin{tabular}{c|c|c|c}
method & label   & order & stab. lower bound\\
  \hline
  \rule{0pt}{12pt}
  backward--forward Euler & \textsc{bfe} & 1 & $\lambda^{\textsc{bfe}}= 1/2$ \\
  implicit--explicit2 & \textsc{imex2} & 2 & $\lambda^{\textsc{imex2}}= 1/2$ \\
  exponential Euler & \textsc{ee} & 1 & $\lambda^{\textsc{ee}}= 1/2$ \\
  exponential RK2 with $\varphi_2$ ($c_2=1$) & \textsc{erk2p2} & 2 & $\lambda^{\textsc{erk2p2}}=1/2$ \\
  exponential RK2 with $\varphi_1$ ($c_2=1$) & \textsc{erk2p1} & 2 & $\lambda^{\textsc{erk2p1}}= 1/3$ \\
  Lawson2a & \textsc{l2a} & 2 & $\lambda^{\textsc{l2a}}= 0.301$ \\
  Lawson2b & \textsc{l2b} & 2 & $\lambda^{\textsc{l2b}}= 0.301$ \\
  Lawson--Euler & \textsc{le} & 1 & $\lambda^{\textsc{le}}= 0.218$ \\
  stabilized Lawson--Euler & \textsc{sle} & 1 & $\lambda^{\textsc{sle}}= 1/(2\ee)\approx 0.184$\\
  stabilized Lawson2 ($\alpha^{\textsc{sl2}}=0.327$) & \textsc{sl2} & 2 & $\lambda^{\textsc{sl2}}= 0.183$
  \end{tabular}
  \caption{Collection of methods, in descending order in terms of unconditional
  stability. Floating point notation means obtained by numerical approximations.}
  \label{tab:stab}
\end{table}

\begin{remark} \rm
The stabilized Lawson--Euler and the stabilized Lawson2 schemes are
\emph{not} exact for linear problems, as opposed to the other exponential
integrators. Moreover, it is straightforward to see that
they are not A-stable (see Figure~\ref{fig:astab} for a plot of the A-stability
regions). Nevertheless, the stability regions contain the whole negative
real axis. Therefore, with both schemes we can treat the Laplacian
operator {\color{black}and small perturbations of it} without incurring a time step size restriction.
\begin{figure}[htb!]
  \centering
  \input{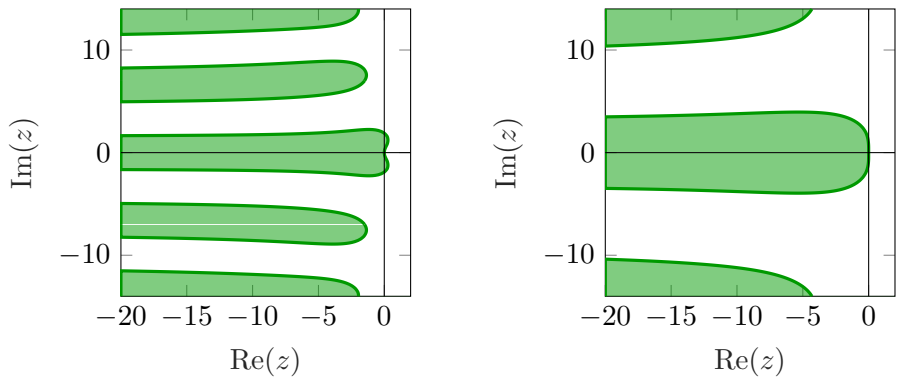}
  \caption{A-stability regions (green) of the \textsc{sle} method
    (left) and of the \textsc{sl2} scheme
    (right).}
  \label{fig:astab}
\end{figure}
\end{remark}

\section{Accelerated exponential methods}\label{sec:methods}
In this section we use the above analysis to
propose a choice for the approximation operator $A$
in a more general context. We
consider in fact the following {\color{black}semilinear} advection-diffusion-reaction equation
\begin{equation} \label{eq:adr}
  \partial_t u(t,\bb x) = \nabla\cdot(a(\bb x) \nabla u(t,\bb x)) +
  \nabla\cdot(\bb b(\bb x) u(t, \bb x)) + r(t, \bx,u(t,\bb x)),
\end{equation}
where $a(\bb x)\in \RR^{d\times d}$ is the diffusion
tensor and $\bb b(\bb x)=(b_\mu(\bb x))_\mu$, for $\mu=1,\ldots,d$,
is a velocity vector field,
with appropriate boundary and initial conditions on the domain
$\Omega\subset\RR^d$. We now rewrite equation~\eqref{eq:adr}
in the following equivalent form
\begin{subequations}\label{eq:fullstab}
\begin{equation} \label{eq:adr_equiv}
    \partial_t u(t,\bb x) = Au(t,\bb x)+g(t,\bb x,u(t,\bb x)),
\end{equation}
where
\begin{equation}\label{eq:linop}
Au(t,\bb x) = \lambda \sum_{\mu=1}^da_{\mu\mu}^{\text{max}} \partial_{x_\mu x_\mu} u(t, \bb x)
+ \bb \beta \cdot \nabla u(t, \bb x)
\end{equation}
and
\begin{equation}\label{eq:nonlin}
g(t, \bx,u(t, \bb x)) =
\nabla\cdot \big((a(\bb x) - \lambda a^{\text{max}}) \nabla u(t,\bb x)\big)
+\nabla\cdot\left((\bb b(\bb x) - \bb \beta)u(t, \bb x)\right)\\
+ r(t, \bx,u(t,\bb x)).
\end{equation}
Here $\lambda\in[0,1]$,
while $a^\text{max}$ and $\bb\beta$ are the diagonal matrix
and the vector with components
\begin{equation}\label{eq:amaxbeta}
a_{\mu \mu}^{\text{max}} = \max_{\bb x} a_{\mu\mu}(\bb x),\quad
  \beta_\mu = \frac{1}{\vert \Omega \vert}\int_{\Omega} \big(b_\mu(\bb x) +
\partial_{x_\mu}a_{\mu\mu}(\bb x) \big) \,d\bb x,  \quad \mu=1,\ldots,d,
\end{equation}
\end{subequations}
respectively.
The idea here is that we extract $\lambda a^{\text{max}}$ from the diffusion
part of the equation as described in the previous section.
In addition, we move into the linear operator also a constant part from
the advection term, which is in general beneficial for exponential integrators.
We have chosen here to use the average
velocity as this physically represents a good constant coefficient
approximation of the advection operator. Moreover,
note that we have to take into
account not only the advection given by $\bb b(\bb x)$ but also the advection
that is implicitly contained in the diffusion operator.

It is possible now to apply to equation~\eqref{eq:adr_equiv}
the methods analysed in section~\ref{sec:linstaban}
for the linear homogeneous diffusion case and formulated
in the appendix for the general case.
The required matrix functions/linear solves for $A$ can be efficiently
computed in a variety of ways (e.g.~using FFT or $\mu$-mode techniques).
More details and an investigation of the performance of these schemes will be
considered in section \ref{sec:numexpstab}.

Note that in order to obtain an efficient scheme the choice of
the parameter $\lambda$ is crucial. In
reference~\cite{WZWS20}, in the
context of IMEX schemes, the authors obtained satisfactory performances
choosing $\lambda$ as small as the stability
condition permits. By following this approach, we also observed
good results in the numerical validation performed in
section~\ref{sec:validation}.
However, additional numerical experiments presented in
section~\ref{sec:numexpstab} show that this
is not necessarily always the best choice. {\color{black}Thus, we wish
to determine $\lambda$ in such a way that the error
of the numerical approximation is as small as possible, while still retaining
efficiency.
Based on numerical evidence of many examples conducted,}
we suggest the following approach. First, we perform a simulation where only a
\textit{small} number of degrees of freedom is used for the space 
discretization {\color{black}(while still capturing the overall physical dynamics
of the equation)}.
Using this simulation we perform a parameter scan for the admissible
values of $\lambda$ in order to determine the one that gives the
smallest error. Only this $\lambda$ is then used in the simulation of the
advection-diffusion-reaction equation. Note that this procedure has
negligible computational cost, since it is performed only for a small-sized
problem. Nevertheless, numerical experiments show that the obtained values
for $\lambda$ still give good approximations to the optimal
ones which correspond to the highest accuracy of the methods
(see, for instance, Figure~\ref{fig:adr2d_diffl}).

\section{Numerical validation of the stability
  constraints}\label{sec:validation}
In this section we show that the stability bounds derived for the model
equation in section~\ref{sec:linstaban} do also apply to more complicated
equations. In particular, we will consider a linear and a nonlinear
one-dimensional diffusion(-reaction) equation with variable coefficients.
\subsection{Linear diffusion}
We start by considering the following one-dimensional linear diffusion
  equation with space dependent coefficients
  \begin{equation}\label{eq:adlinearzero}
    \left\{
    \begin{aligned}
      \partial_t u(t,x)&=a(x)\partial_{xx}u(t,x),&x\in(-\pi,\pi),\ t\in[0,T],\\
      u(0,x)&=\sin x ,
    \end{aligned}\right.
  \end{equation}
  subject to periodic boundary conditions.
  In the experiments we use $a(x)=1+10\sin^2 x $.
As described above, we then rewrite equation~\eqref{eq:adlinearzero} as
\begin{equation} \label{eq:ex0stab}
  \partial_tu(t,x)=\underbrace{\lambda a^{\mathrm{max}}\partial_{xx} u(t,x)}_{A u(t,x)}+
  \underbrace{(a(x)-\lambda a^{\mathrm{max}})\partial_{xx}u(t,x)}_{g(x,u(t,x))},
\end{equation}
with $\lambda\in[0,1]$.

The structure of the equation allows for an effective discretization in
space by means of a Fourier spectral technique. In particular, we denote with
$N$ the number of Fourier modes.
Then, the temporal schemes studied in section~\ref{sec:linstaban} and resumed
in the appendix can be applied in a straightforward manner,
with the computation of derivatives and matrix functions
by pointwise operations on Fourier coefficients.

Here, we verify that the theoretical lower bounds for $\lambda$ found in
section~\ref{sec:linstaban} also
apply to this space dependent coefficients diffusion equation.
The actual simulations have been carried out with $N=2^{12}$ and $T=1/40$, and
the achieved errors for different $\lambda$ and varying number of time steps
$m=2^\ell$, with $\ell=4,6,8,\ldots,14$,
have been measured at the final time $T$ in the infinity norm,
relatively with respect to the exact solution of
equation~\eqref{eq:adlinearzero}.
The results are collected in {\color{black}Figures~\ref{fig:ad1dex0_diffdeltaa} 
and~\ref{fig:ad1dex0_diffdeltab}}. First of all,
we observe that all the considered exponential methods show the expected order
of convergence (in particular also the newly derived stabilized Lawson
schemes).
Then, we also clearly see that, as $\lambda$ decreases, some
methods fail to be unconditionally stable,
{\color{black} producing large errors when decreasing the time step size
  (for visualization reasons, we restrict the plot
  to values below $10^{-1}$)}.
In particular, if we compare with
the bounds resumed in Table~\ref{tab:stab}, we observe that the values found
can be applied sharply also to the case of equation~\eqref{eq:adlinearzero}.
Indeed, up to $\lambda=1/2$ all the methods are stable. Then,
if we further decrease $\lambda$, some methods start to blow up when
incrementing the number of time steps, in accordance to what presented in
Table~\ref{tab:stab}
({\color{black}for instance, at $\lambda=1/3$ we clearly see that the exponential
Euler scheme and the \textsc{erk2p2} method are not unconditionally stable}).
Finally, as predicted by the linear analysis, for $\lambda<\lambda^{\textsc{sl2}}$
all the schemes {\color{black}lose the unconditional stability property}.

\begin{figure}[!hp]
  \centering
%
%
\definecolor{mycolor1}{rgb}{0.00000,0.44700,0.74100}%
\definecolor{mycolor2}{rgb}{255,0,0}%
\definecolor{mycolor3}{rgb}{0.92900,0.69400,0.12500}%
\definecolor{mycolor4}{rgb}{0.49400,0.18400,0.55600}%
\definecolor{mycolor5}{rgb}{0.46600,0.67400,0.18800}%
\definecolor{mycolor6}{rgb}{0.30100,0.74500,0.93300}%
\definecolor{mycolor7}{rgb}{0.63500,0.07800,0.18400}%
\definecolor{mycolor8}{rgb}{255,0,205}%
\begin{tikzpicture}

\begin{axis}[%
width=2.251in,
height=3.545in,
at={(0.948in,0.573in)},
scale only axis,
xmode=log,
xmin=10,
xmax=20000,
xminorticks=true,
xlabel style={font=\color{white!15!black}},
xlabel={$m$},
ymode=log,
ymin=1e-11,
ymax=1e-1,
yminorticks=true,
ylabel style={font=\color{white!15!black}},
ylabel={Error},
axis background/.style={fill=white},
title={$\boldsymbol{\lambda=1}$},
]
\addplot [color=mycolor1, line width=1.1pt, mark size=3.5pt, mark=o, mark options={solid, mycolor1}]
  table[row sep=crcr]{%
16	0.00423707305496343\\
64.0000000000001	0.00108108655299582\\
256	0.000270511481595112\\
1024	6.7626916419749e-05\\
4096	1.69063292842301e-05\\
16384	4.22655409102637e-06\\
};

\addplot [color=mycolor2, line width=1.1pt, mark size=3.5pt, mark=x, mark options={solid, mycolor2}]
  table[row sep=crcr]{%
16	0.00916012013080849\\
64.0000000000001	0.00288149591423319\\
256	0.000793418029777495\\
1024	0.000204419218019381\\
4096	5.14937168754167e-05\\
16384	1.2897202022864e-05\\
};

\addplot [color=mycolor3, line width=1.1pt, mark size=3.5pt, mark=+, mark options={solid, mycolor3}]
  table[row sep=crcr]{%
16	0.00192512919614741\\
64	0.000490216178041984\\
256	0.000122988681635148\\
1024	3.07721099012625e-05\\
4096	7.69454863964331e-06\\
16384	1.92373412585277e-06\\
};

\addplot [color=mycolor4, line width=1.1pt, mark size=3.5pt, mark=diamond, mark options={solid, mycolor4}]
  table[row sep=crcr]{%
16	0.000285071986247785\\
64	2.32555533473746e-05\\
256	1.59420350459551e-06\\
1024	1.02312523570604e-07\\
4096	6.44042420217275e-09\\
16384	4.05011554485088e-10\\
};

\addplot [color=mycolor5, line width=1.1pt, mark size=3.5pt, mark=triangle, mark options={solid, mycolor5}]
  table[row sep=crcr]{%
16	0.000645242563331627\\
64	5.33674120768873e-05\\
256	3.69226405602089e-06\\
1024	2.37873737189419e-07\\
4096	1.49880003548267e-08\\
16384	9.40426068901151e-10\\
};

\addplot [color=mycolor6, line width=1.1pt, mark size=3.5pt, mark=triangle, mark options={solid, rotate=270, mycolor6}]
  table[row sep=crcr]{%
16	0.000377985069814101\\
64.0000000000001	3.14910144760689e-05\\
256	2.17859068745e-06\\
1024	1.40244963567249e-07\\
4096.00000000001	8.83486629638264e-09\\
16384	5.55039286402483e-10\\
};

\addplot [color=mycolor7, line width=1.1pt, mark size=3.5pt, mark=triangle, mark options={solid, rotate=90, mycolor7}]
  table[row sep=crcr]{%
16	0.00111714158295021\\
64.0000000000001	0.000125324878516901\\
256	9.9318195277156e-06\\
1024	6.61272902653368e-07\\
4096.00000000001	4.19523796011746e-08\\
16384	2.63283396862452e-09\\
};

\addplot [color=mycolor8, line width=1.1pt, mark size = 2.5pt, mark=triangle, mark options={solid, rotate=180, mycolor8}]
  table[row sep=crcr]{%
16	0.000830893229003044\\
64.0000000000001	0.000100305683926654\\
256	8.19746560465537e-06\\
1024	5.50790354654275e-07\\
4096.00000000001	3.50180862789402e-08\\
16384	2.19896139644816e-09\\
};
\end{axis}

\end{tikzpicture}%
%
%
\definecolor{mycolor1}{rgb}{0.00000,0.44700,0.74100}%
\definecolor{mycolor2}{rgb}{255,0,0}%
\definecolor{mycolor3}{rgb}{0.92900,0.69400,0.12500}%
\definecolor{mycolor4}{rgb}{0.49400,0.18400,0.55600}%
\definecolor{mycolor5}{rgb}{0.46600,0.67400,0.18800}%
\definecolor{mycolor6}{rgb}{0.30100,0.74500,0.93300}%
\definecolor{mycolor7}{rgb}{0.63500,0.07800,0.18400}%
\definecolor{mycolor8}{rgb}{255,0,205}%
\begin{tikzpicture}

\begin{axis}[%
width=2.251in,
height=3.545in,
at={(0.948in,0.573in)},
scale only axis,
xmode=log,
xmin=10,
xmax=20000,
xminorticks=true,
xlabel style={font=\color{white!15!black}},
xlabel={$m$},
ymode=log,
ymin=1e-11,
ymax=1e-1,
yminorticks=true,
ylabel style={font=\color{white!15!black}},
ylabel={Error},
axis background/.style={fill=white},
title={$\boldsymbol{\lambda=1/2}$},
]
\addplot [color=mycolor1, line width=1.1pt, mark size=3.5pt, mark=o, mark options={solid, mycolor1}]
  table[row sep=crcr]{%
16	0.00199874808226911\\
64.0000000000001	0.000494878466287692\\
256	0.000123281833768064\\
1024	3.07904737083712e-05\\
4096	7.69569703774469e-06\\
16384	1.92380591132402e-06\\
};

\addplot [color=mycolor2, line width=1.1pt, mark size=3.5pt, mark=x, mark options={solid, mycolor2}]
  table[row sep=crcr]{%
16	0.0035982974109011\\
64.0000000000001	0.000999922408098618\\
256	0.000258527584778083\\
1024	6.51966856462329e-05\\
4096	1.63343244291428e-05\\
16384	4.0857660323972e-06\\
};

\addplot [color=mycolor3, line width=1.1pt, mark size=3.5pt, mark=+, mark options={solid, mycolor3}]
  table[row sep=crcr]{%
16	0.00120770813413409\\
64.0000000000001	0.000298735157089986\\
256	7.44879541391514e-05\\
1024	1.86098321406297e-05\\
4096	4.65172510406573e-06\\
16384	1.16291136473732e-06\\
};

\addplot [color=mycolor4, line width=1.1pt, mark size=3.5pt, mark=diamond, mark options={solid, mycolor4}]
  table[row sep=crcr]{%
16	6.85721873297132e-05\\
64.0000000000001	5.06678401438467e-06\\
256	3.33800872619692e-07\\
1024	2.11636235540015e-08\\
4096.00000000001	1.32924830256509e-09\\
16384	8.48747407446946e-11\\
};

\addplot [color=mycolor5, line width=1.1pt, mark size=3.5pt, mark=triangle, mark options={solid, mycolor5}]
  table[row sep=crcr]{%
16	0.000171252803424007\\
64.0000000000001	1.27615012906381e-05\\
256	8.45833116706716e-07\\
1024	5.37362489935689e-08\\
4096.00000000001	3.37431910856005e-09\\
16384	2.12813851818298e-10\\
};

\addplot [color=mycolor6, line width=1.1pt, mark size=3.5pt, mark=triangle, mark options={solid, rotate=270, mycolor6}]
  table[row sep=crcr]{%
16	9.50728081660571e-05\\
64.0000000000001	7.12568566353044e-06\\
256	4.72053282975986e-07\\
1024	2.99780154620206e-08\\
4096.00000000001	1.88297900623374e-09\\
16384	1.1951291446954e-10\\
};

\addplot [color=mycolor7, line width=1.1pt, mark size=3.5pt, mark=triangle, mark options={solid, rotate=90, mycolor7}]
  table[row sep=crcr]{%
16	0.000240845615881032\\
64.0000000000001	2.08169751824487e-05\\
256	1.43858089228628e-06\\
1024	9.22345710450661e-08\\
4096.00000000001	5.80288048845724e-09\\
16384	3.64455078610185e-10\\
};

\addplot [color=mycolor8, line width=1.1pt, mark size=3.5pt, mark=triangle, mark options={solid, rotate=180, mycolor8}]
  table[row sep=crcr]{%
16	0.000141591106454595\\
64.0000000000001	1.35218166096352e-05\\
256	9.65137253713493e-07\\
1024	6.24016250944227e-08\\
4096.00000000001	3.93454913691426e-09\\
16384	2.47569585839744e-10\\
};

\end{axis}
\end{tikzpicture}%
%
%
\definecolor{mycolor1}{rgb}{0.00000,0.44700,0.74100}%
\definecolor{mycolor2}{rgb}{255,0,0}%
\definecolor{mycolor3}{rgb}{0.92900,0.69400,0.12500}%
\definecolor{mycolor4}{rgb}{0.49400,0.18400,0.55600}%
\definecolor{mycolor5}{rgb}{0.46600,0.67400,0.18800}%
\definecolor{mycolor6}{rgb}{0.30100,0.74500,0.93300}%
\definecolor{mycolor7}{rgb}{0.63500,0.07800,0.18400}%
\definecolor{mycolor8}{rgb}{255,0,205}%
\begin{tikzpicture}

\begin{axis}[%
width=2.251in,
height=3.545in,
at={(0.948in,0.573in)},
scale only axis,
xmode=log,
xmin=10,
xmax=20000,
xminorticks=true,
xlabel style={font=\color{white!15!black}},
xlabel={$m$},
ymode=log,
ymin=1e-11,
ymax=1e-1,
yminorticks=true,
ylabel style={font=\color{white!15!black}},
ylabel={Error},
axis background/.style={fill=white},
title={$\boldsymbol{\lambda=1/3}$},
]
\addplot [color=mycolor1, line width=1.1pt, mark size=3.5pt, mark=o, mark options={solid, mycolor1}]
  table[row sep=crcr]{%
16	0.00121295800258889\\
64.0000000000001	0.000298000405612343\\
256	7.41330536453391e-05\\
1024	1.85095261311084e-05\\
4096	4.62588571827088e-06\\
16384	1.15638018180848e-06\\
};

\addplot [color=mycolor2, line width=1.1pt, mark size=3.5pt, mark=x, mark options={solid, mycolor2}]
  table[row sep=crcr]{%
16	0.00199926133822241\\
64.0000000000001	0.000530120548956881\\
256	0.000134752360774442\\
1024	3.38298406721121e-05\\
4096	8.46638257189193e-06\\
16384	2.11715377704209e-06\\
};

\addplot [color=mycolor3, line width=1.1pt, mark size=3.5pt, mark=+, mark options={solid, mycolor3}]
  table[row sep=crcr]{%
16	0.00161332835038153\\
35.2246181271557	12.5892541179417\\
};

\addplot [color=mycolor4, line width=1.1pt, mark size=3.5pt, mark=diamond, mark options={solid, mycolor4}]
  table[row sep=crcr]{%
16	2.66878040455841e-05\\
21.2691463167461	12.5892541179417\\
};

\addplot [color=mycolor5, line width=1.1pt, mark size=3.5pt, mark=triangle, mark options={solid, mycolor5}]
  table[row sep=crcr]{%
16	7.38266119038114e-05\\
64.0000000000001	5.25156656449825e-06\\
256	3.41977607644209e-07\\
1024	2.16136344292725e-08\\
4096.00000000001	1.35629883197778e-09\\
16384	8.64559565894357e-11\\
};

\addplot [color=mycolor6, line width=1.1pt, mark size=3.5pt, mark=triangle, mark options={solid, rotate=270, mycolor6}]
  table[row sep=crcr]{%
16	3.85836099047089e-05\\
64.0000000000001	2.75703060055272e-06\\
256	1.79499449252316e-07\\
1024	1.13432401238343e-08\\
4096.00000000001	7.12548798654149e-10\\
16384	4.62212524137113e-11\\
};

\addplot [color=mycolor7, line width=1.1pt, mark size=3.5pt, mark=triangle, mark options={solid, rotate=90, mycolor7}]
  table[row sep=crcr]{%
16	8.89883775933262e-05\\
64.0000000000001	6.86248790932651e-06\\
256	4.55219890560559e-07\\
1024	2.88895126282396e-08\\
4096.00000000001	1.81391621319197e-09\\
16384	1.14822374544655e-10\\
};

\addplot [color=mycolor8, line width=1.1pt, mark size=3.5pt, mark=triangle, mark options={solid, rotate=180, mycolor8}]
  table[row sep=crcr]{%
16	3.73941053200776e-05\\
64.0000000000001	3.32665322365956e-06\\
256	2.29673115524828e-07\\
1024	1.47246737377834e-08\\
4096.00000000001	9.2734665170747e-10\\
16384	5.92804030130875e-11\\
};

\end{axis}
\end{tikzpicture}%
%
%
\definecolor{mycolor1}{rgb}{0.00000,0.44700,0.74100}%
\definecolor{mycolor2}{rgb}{255,0,0}%
\definecolor{mycolor3}{rgb}{0.92900,0.69400,0.12500}%
\definecolor{mycolor4}{rgb}{0.49400,0.18400,0.55600}%
\definecolor{mycolor5}{rgb}{0.46600,0.67400,0.18800}%
\definecolor{mycolor6}{rgb}{0.30100,0.74500,0.93300}%
\definecolor{mycolor7}{rgb}{0.63500,0.07800,0.18400}%
\definecolor{mycolor8}{rgb}{255,0,205}%
\begin{tikzpicture}

\begin{axis}[%
width=2.251in,
height=3.545in,
at={(0.948in,0.573in)},
scale only axis,
xmode=log,
xmin=10,
xmax=20000,
xminorticks=true,
xlabel style={font=\color{white!15!black}},
xlabel={$m$},
ymode=log,
ymin=1e-11,
ymax=1e-1,
yminorticks=true,
ylabel style={font=\color{white!15!black}},
ylabel={Error},
axis background/.style={fill=white},
title={$\boldsymbol{\lambda=0.301}$},
legend style={at={(0.07,0.03)}, font=\scriptsize, anchor=south west, legend cell align=left, align=left, draw=white!15!black, legend columns=1}
]
\addplot [color=mycolor1, line width=1.1pt, mark size=3.5pt, mark=o, mark options={solid, mycolor1}]
  table[row sep=crcr]{%
16	0.00105865256664539\\
64.0000000000001	0.000259734435862798\\
256	6.45967755646223e-05\\
1024	1.61275468409584e-05\\
4096	4.03052479468083e-06\\
16384	1.0075483554887e-06\\
};
\addlegendentry{\textsc{sle}}

\addplot [color=mycolor2, line width=1.1pt, mark size=3.5pt, mark=x, mark options={solid, mycolor2}]
  table[row sep=crcr]{%
16	0.00171275442711519\\
64.0000000000001	0.000450175723005691\\
256	0.000114113504167043\\
1024	2.86283392938143e-05\\
4096	7.16337506536909e-06\\
16384	1.7912378100778e-06\\
};
\addlegendentry{\textsc{le}}

\addplot [color=mycolor3, line width=1.1pt, mark size=3.5pt, mark=+, mark options={solid, mycolor3}]
  table[row sep=crcr]{%
16	0.00169184749768018\\
26.1106992760665	12.5892541179417\\
};
\addlegendentry{\textsc{ee}}

\addplot [color=mycolor4, line width=1.1pt, mark size=3.5pt, mark=diamond, mark options={solid, mycolor4}]
  table[row sep=crcr]{%
16	0.000312437822491241\\
19.202841613847	12.5892541179417\\
};
\addlegendentry{\textsc{erk2p2}}

\addplot [color=mycolor5, line width=1.1pt, mark size=3.5pt, mark=triangle, mark options={solid, mycolor5}]
  table[row sep=crcr]{%
16	5.92744077547448e-05\\
64.0000000000001	4.17119958089086e-06\\
256	2.70570537808496e-07\\
1024	1.70822184598647e-08\\
4096.00000000001	1.07195050559046e-09\\
16384	6.86486098874435e-11\\
};
\addlegendentry{\textsc{sl2}}

\addplot [color=mycolor6, line width=1.1pt, mark size=3.5pt, mark=triangle, mark options={solid, rotate=270, mycolor6}]
  table[row sep=crcr]{%
16	3.03387434736157e-05\\
64.0000000000001	0.000502269585121561\\
75.9224557673045	12.5892541179416\\
};
\addlegendentry{\textsc{erk2p1}}

\addplot [color=mycolor7, line width=1.1pt, mark size=3.5pt, mark=triangle, mark options={solid, rotate=90, mycolor7}]
  table[row sep=crcr]{%
16	6.84061940092523e-05\\
64.0000000000001	5.15143340858422e-06\\
256	3.39170052380535e-07\\
1024	2.14843112186925e-08\\
4096.00000000001	1.34863686511182e-09\\
16384	8.55555859162807e-11\\
};
\addlegendentry{\textsc{l2a}}

\addplot [color=mycolor8, line width=1.1pt, mark size=3.5pt, mark=triangle, mark options={solid, rotate=180, mycolor8}]
  table[row sep=crcr]{%
16	2.78798968996643e-05\\
64.0000000000001	2.20032062329372e-06\\
256	1.51509748631173e-07\\
1024	9.70690040705074e-09\\
4096.00000000001	6.11587896512005e-10\\
16384	5.04437990933306e-11\\
};
\addlegendentry{\textsc{l2b}}

\end{axis}
\end{tikzpicture}%
  \caption{Solution of problem~\eqref{eq:adlinearzero} rewritten as
           equation~\eqref{eq:ex0stab},
           with different schemes, decreasing $\lambda$ and varying number of
           time steps $m$. 
            {\color{black}The common legend is displayed in the bottom right
              plot. Missing marks mean that the error is larger than
               $10^{-1}$, i.e., the method fails to be unconditionally
            stable.}}
  \label{fig:ad1dex0_diffdeltaa}
\end{figure}
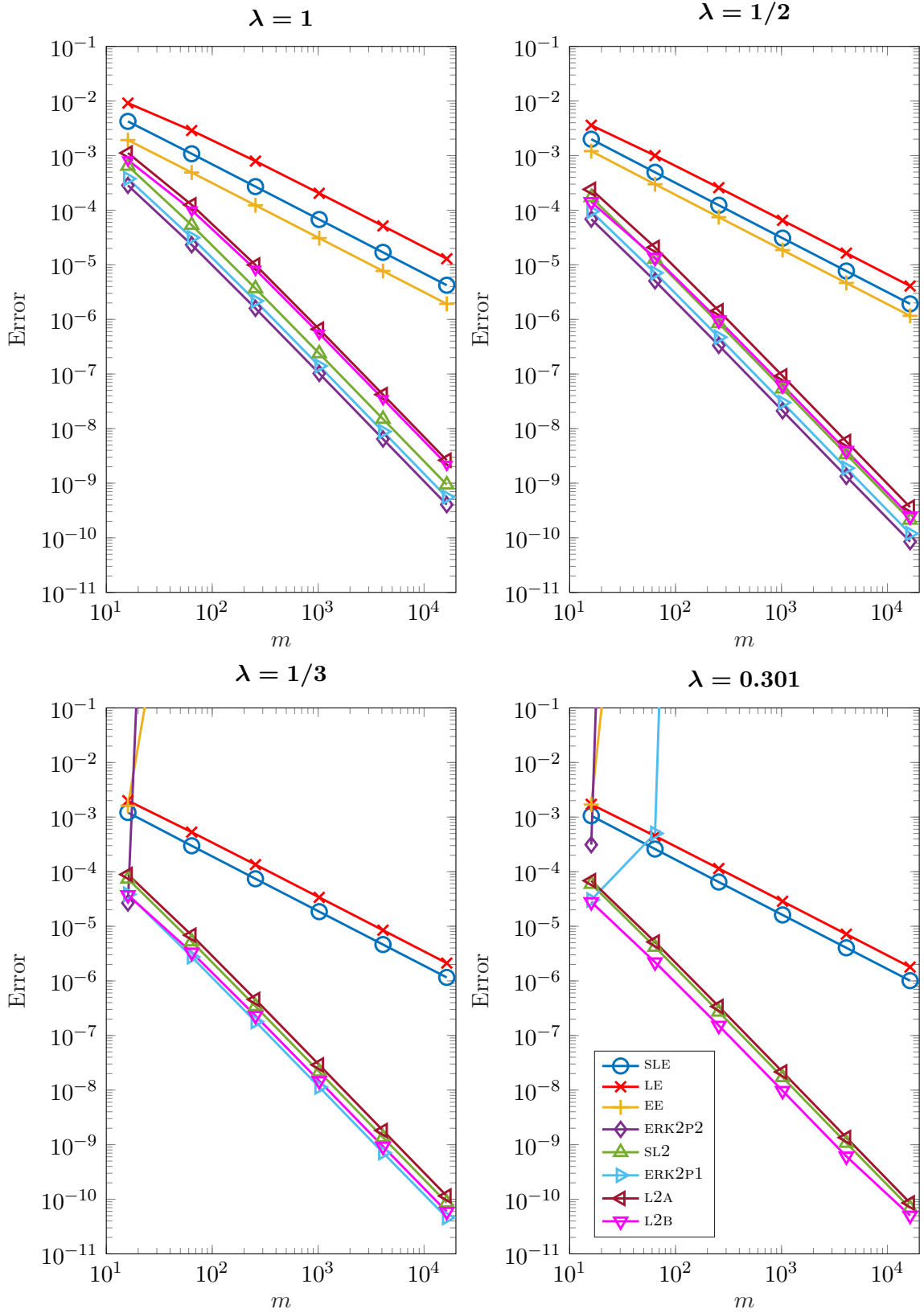
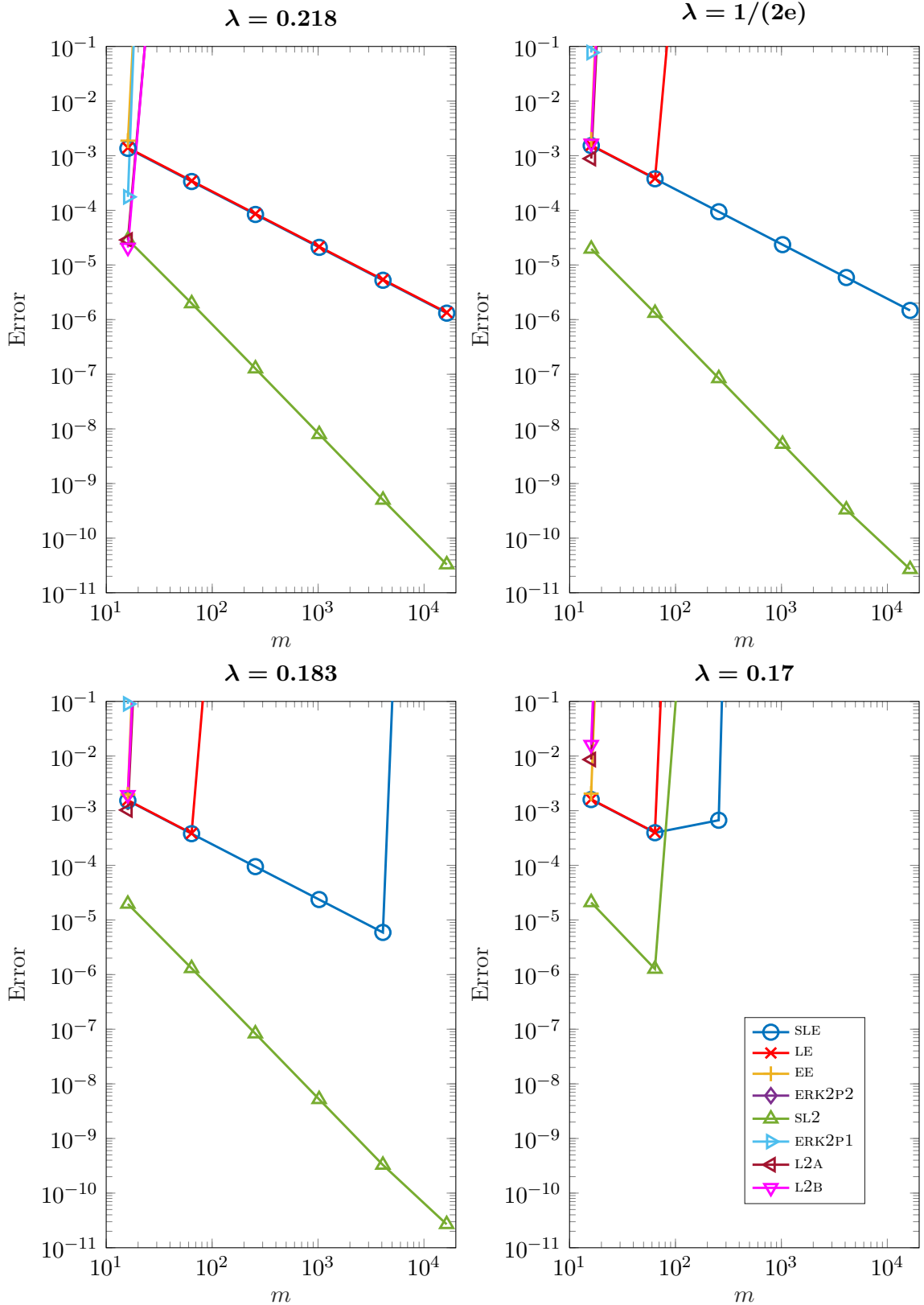
\begin{figure}[!hp]
  \centering
%
%
\definecolor{mycolor1}{rgb}{0.00000,0.44700,0.74100}%
\definecolor{mycolor2}{rgb}{255,0,0}%
\definecolor{mycolor3}{rgb}{0.92900,0.69400,0.12500}%
\definecolor{mycolor4}{rgb}{0.49400,0.18400,0.55600}%
\definecolor{mycolor5}{rgb}{0.46600,0.67400,0.18800}%
\definecolor{mycolor6}{rgb}{0.30100,0.74500,0.93300}%
\definecolor{mycolor7}{rgb}{0.63500,0.07800,0.18400}%
\definecolor{mycolor8}{rgb}{255,0,205}%
\begin{tikzpicture}

\begin{axis}[%
width=2.251in,
height=3.545in,
at={(0.948in,0.573in)},
scale only axis,
xmode=log,
xmin=10,
xmax=20000,
xminorticks=true,
xlabel style={font=\color{white!15!black}},
xlabel={$m$},
ymode=log,
ymin=1e-11,
ymax=1e-1,
yminorticks=true,
ylabel style={font=\color{white!15!black}},
ylabel={Error},
axis background/.style={fill=white},
title={$\boldsymbol{\lambda=0.218}$},
]
\addplot [color=mycolor1, line width=1.1pt, mark size=3.5pt, mark=o, mark options={solid, mycolor1}]
  table[row sep=crcr]{%
16	0.00136157072083516\\
64.0000000000001	0.000337072688365754\\
256	8.40641856715537e-05\\
1024	2.10033775755172e-05\\
4096	5.25007936984525e-06\\
16384	1.31249796785149e-06\\
};

\addplot [color=mycolor2, line width=1.1pt, mark size=3.5pt, mark=x, mark options={solid, mycolor2}]
  table[row sep=crcr]{%
16	0.00139947356560928\\
64.0000000000001	0.000346465068523869\\
256	8.64071571382091e-05\\
1024	2.15888010434998e-05\\
4096	5.39641496775279e-06\\
16384	1.34907940393613e-06\\
};

\addplot [color=mycolor3, line width=1.1pt, mark size=3.5pt, mark=+, mark options={solid, mycolor3}]
  table[row sep=crcr]{%
16	0.0018846696308818\\
20.1730376100773	12.5892541179418\\
};


\addplot [color=mycolor5, line width=1.1pt, mark size=3.5pt, mark=triangle, mark options={solid, mycolor5}]
  table[row sep=crcr]{%
16	2.90161154184961e-05\\
64.0000000000001	1.97520127008208e-06\\
256	1.26690240540664e-07\\
1024	7.97407705463469e-09\\
4096.00000000001	5.00763660898246e-10\\
16384	3.28645832925689e-11\\
};

\addplot [color=mycolor6, line width=1.1pt, mark size=3.5pt, mark=triangle, mark options={solid, rotate=270, mycolor6}]
  table[row sep=crcr]{%
16	0.000176240400696085\\
19.8757222840139	12.5892541179416\\
};

\addplot [color=mycolor7, line width=1.1pt, mark size=3.5pt, mark=triangle, mark options={solid, rotate=90, mycolor7}]
  table[row sep=crcr]{%
16	2.87034351417647e-05\\
28.7596596347026	12.5892541179417\\
};

\addplot [color=mycolor8, line width=1.1pt, mark size=3.5pt, mark=triangle, mark options={solid, rotate=180, mycolor8}]
  table[row sep=crcr]{%
16	2.08618740774127e-05\\
28.505815317335	12.5892541179417\\
};

\end{axis}
\end{tikzpicture}%
%
%
\definecolor{mycolor1}{rgb}{0.00000,0.44700,0.74100}%
\definecolor{mycolor2}{rgb}{255,0,0}%
\definecolor{mycolor3}{rgb}{0.92900,0.69400,0.12500}%
\definecolor{mycolor4}{rgb}{0.49400,0.18400,0.55600}%
\definecolor{mycolor5}{rgb}{0.46600,0.67400,0.18800}%
\definecolor{mycolor6}{rgb}{0.30100,0.74500,0.93300}%
\definecolor{mycolor7}{rgb}{0.63500,0.07800,0.18400}%
\definecolor{mycolor8}{rgb}{255,0,205}%
\begin{tikzpicture}

\begin{axis}[%
width=2.251in,
height=3.545in,
at={(0.948in,0.573in)},
scale only axis,
xmode=log,
xmin=10,
xmax=20000,
xminorticks=true,
xlabel style={font=\color{white!15!black}},
xlabel={$m$},
ymode=log,
ymin=1e-11,
ymax=1e-1,
yminorticks=true,
ylabel style={font=\color{white!15!black}},
ylabel={Error},
axis background/.style={fill=white},
title={$\boldsymbol{\lambda=1/(2\mathrm{e)}}$},
]
\addplot [color=mycolor1, line width=1.1pt, mark size=3.5pt, mark=o, mark options={solid, mycolor1}]
  table[row sep=crcr]{%
16	0.00152794454512032\\
64.0000000000001	0.000378039213948396\\
256	9.42669766767329e-05\\
1024	2.35516531695002e-05\\
4096	5.8869969619872e-06\\
16384	1.47171791252177e-06\\
};

\addplot [color=mycolor2, line width=1.1pt, mark size=3.5pt, mark=x, mark options={solid, mycolor2}]
  table[row sep=crcr]{%
16	0.0015547390329062\\
64	0.000384713767930697\\
103.69284500531	12.5892541179417\\
};

\addplot [color=mycolor3, line width=1.1pt, mark size=3.5pt, mark=+, mark options={solid, mycolor3}]
  table[row sep=crcr]{%
16	0.00197588065669322\\
19.2237166021692	12.5892541179416\\
};


\addplot [color=mycolor5, line width=1.1pt, mark size=3.5pt, mark=triangle, mark options={solid, mycolor5}]
  table[row sep=crcr]{%
16	1.97029392809251e-05\\
64.0000000000001	1.3179094544265e-06\\
256	8.40610072683831e-08\\
1024	5.28339765973368e-09\\
4096.00000000001	3.32143348340766e-10\\
16384	2.69498637986954e-11\\
};

\addplot [color=mycolor6, line width=1.1pt, mark size=3.5pt, mark=triangle, mark options={solid, rotate=270, mycolor6}]
  table[row sep=crcr]{%
16	0.0772531253519745\\
17.2359075084677	12.5892541179417\\
};

\addplot [color=mycolor7, line width=1.1pt, mark size=3.5pt, mark=triangle, mark options={solid, rotate=90, mycolor7}]
  table[row sep=crcr]{%
16	0.000884466455828353\\
20.2184493312236	12.5892541179417\\
};

\addplot [color=mycolor8, line width=1.1pt, mark size=3.5pt, mark=triangle, mark options={solid, rotate=180, mycolor8}]
  table[row sep=crcr]{%
16	0.00167367680968051\\
19.8211180961655	12.5892541179418\\
};

\end{axis}
\end{tikzpicture}%
%
%
\definecolor{mycolor1}{rgb}{0.00000,0.44700,0.74100}%
\definecolor{mycolor2}{rgb}{255,0,0}%
\definecolor{mycolor3}{rgb}{0.92900,0.69400,0.12500}%
\definecolor{mycolor4}{rgb}{0.49400,0.18400,0.55600}%
\definecolor{mycolor5}{rgb}{0.46600,0.67400,0.18800}%
\definecolor{mycolor6}{rgb}{0.30100,0.74500,0.93300}%
\definecolor{mycolor7}{rgb}{0.63500,0.07800,0.18400}%
\definecolor{mycolor8}{rgb}{255,0,205}%
\begin{tikzpicture}

\begin{axis}[%
width=2.251in,
height=3.545in,
at={(0.959in,0.568in)},
scale only axis,
xmode=log,
xmin=10,
xmax=20000,
xminorticks=true,
xlabel style={font=\color{white!15!black}},
xlabel={$m$},
ymode=log,
ymin=1e-11,
ymax=1e-1,
yminorticks=true,
ylabel style={font=\color{white!15!black}},
ylabel={Error},
axis background/.style={fill=white},
title={$\boldsymbol{\lambda=0.183}$},
]
\addplot [color=mycolor1, line width=1.1pt, mark size=3.5pt, mark=o, mark options={solid, mycolor1}]
  table[row sep=crcr]{%
16	0.00153253457142279\\
64.0000000000001	0.000379169465279559\\
256	9.45484701480841e-05\\
1024	2.36219598538853e-05\\
4096	5.90456948355975e-06\\
5553.53686585094	12.5892541179417\\
};

\addplot [color=mycolor2, line width=1.1pt, mark size=3.5pt, mark=x, mark options={solid, mycolor2}]
  table[row sep=crcr]{%
16	0.00155905043130211\\
64	0.000385775705594358\\
100.632075044529	12.5892541179417\\
};

\addplot [color=mycolor3, line width=1.1pt, mark size=3.5pt, mark=+, mark options={solid, mycolor3}]
  table[row sep=crcr]{%
16	0.00197986382978937\\
19.2014297842096	12.5892541179417\\
};


\addplot [color=mycolor5, line width=1.1pt, mark size=3.5pt, mark=triangle, mark options={solid, mycolor5}]
  table[row sep=crcr]{%
16	1.9658448736235e-05\\
64.0000000000001	1.30179959270465e-06\\
256	8.30198161843351e-08\\
1024	5.21773069789467e-09\\
4096.00000000001	3.28029229696882e-10\\
16384	2.68909958401947e-11\\
};

\addplot [color=mycolor6, line width=1.1pt, mark size=3.5pt, mark=triangle, mark options={solid, rotate=270, mycolor6}]
  table[row sep=crcr]{%
16	0.0900547383212365\\
17.1993533891795	12.5892541179417\\
};

\addplot [color=mycolor7, line width=1.1pt, mark size=3.5pt, mark=triangle, mark options={solid, rotate=90, mycolor7}]
  table[row sep=crcr]{%
16	0.00102741527207708\\
20.0776675215351	12.5892541179417\\
};

\addplot [color=mycolor8, line width=1.1pt, mark size=3.5pt, mark=triangle, mark options={solid, rotate=180, mycolor8}]
  table[row sep=crcr]{%
16	0.0019448457363291\\
19.7762287252291	12.5892541179417\\
};

\end{axis}
\end{tikzpicture}%
%
%
\definecolor{mycolor1}{rgb}{0.00000,0.44700,0.74100}%
\definecolor{mycolor2}{rgb}{255,0,0}%
\definecolor{mycolor3}{rgb}{0.92900,0.69400,0.12500}%
\definecolor{mycolor4}{rgb}{0.49400,0.18400,0.55600}%
\definecolor{mycolor5}{rgb}{0.46600,0.67400,0.18800}%
\definecolor{mycolor6}{rgb}{0.30100,0.74500,0.93300}%
\definecolor{mycolor7}{rgb}{0.63500,0.07800,0.18400}%
\definecolor{mycolor8}{rgb}{255,0,205}%
\begin{tikzpicture}

\begin{axis}[%
width=2.251in,
height=3.545in,
at={(0.959in,0.568in)},
scale only axis,
xmode=log,
xmin=10,
xmax=20000,
xminorticks=true,
xlabel style={font=\color{white!15!black}},
xlabel={$m$},
ymode=log,
ymin=1e-11,
ymax=1e-1,
yminorticks=true,
ylabel style={font=\color{white!15!black}},
ylabel={Error},
axis background/.style={fill=white},
title={$\boldsymbol{\lambda=0.17}$},
legend style={at={(0.5,0.08)}, font=\scriptsize, anchor=south west, legend cell align=left, align=left, draw=white!15!black, legend columns=1}
]
\addplot [color=mycolor1, line width=1.1pt, mark size=3.5pt, mark=o, mark options={solid, mycolor1}]
  table[row sep=crcr]{%
16	0.00159603050915978\\
64	0.000394805182136802\\
256	0.000667843038369431\\
293.573294034139	12.5892541179417\\
};
\addlegendentry{\textsc{sle}}

\addplot [color=mycolor2, line width=1.1pt, mark size=3.5pt, mark=x, mark options={solid, mycolor2}]
  table[row sep=crcr]{%
16	0.00161883979208275\\
64	0.000400490771727264\\
80.9714738716388	12.5892541179417\\
};
\addlegendentry{\textsc{le}}

\addplot [color=mycolor3, line width=1.1pt, mark size=3.5pt, mark=+, mark options={solid, mycolor3}]
  table[row sep=crcr]{%
16	0.00205348101133825\\
18.9094163972867	12.5892541179418\\
};
\addlegendentry{\textsc{ee}}

\addplot [color=mycolor4, line width=1.1pt, mark size=3.5pt, mark=diamond, mark options={solid, mycolor4}]
  table[row sep=crcr]{%
16	1e0\\
};
\addlegendentry{\textsc{erk2p2}}

\addplot [color=mycolor5, line width=1.1pt, mark size=3.5pt, mark=triangle, mark options={solid, mycolor5}]
  table[row sep=crcr]{%
16	2.09610027254457e-05\\
64.0000000000001	1.26704814926524e-06\\
122.022374875128	12.5892541179417\\
};
\addlegendentry{\textsc{sl2}}

\addplot [color=mycolor6, line width=1.1pt, mark size=3.5pt, mark=triangle, mark options={solid, rotate=270, mycolor6}]
  table[row sep=crcr]{%
16	1e0\\
};
\addlegendentry{\textsc{erk2p1}}

\addplot [color=mycolor7, line width=1.1pt, mark size=3.5pt, mark=triangle, mark options={solid, rotate=90, mycolor7}]
  table[row sep=crcr]{%
16	0.00863508545848884\\
18.5491696915454	12.5892541179417\\
};
\addlegendentry{\textsc{l2a}}

\addplot [color=mycolor8, line width=1.1pt, mark size=3.5pt, mark=triangle, mark options={solid, rotate=180, mycolor8}]
  table[row sep=crcr]{%
16	0.0161177106677327\\
18.2974047681636	12.5892541179417\\
};
\addlegendentry{\textsc{l2b}}

\end{axis}
\end{tikzpicture}%
  \caption{Solution of problem~\eqref{eq:adlinearzero} rewritten as
           equation~\eqref{eq:ex0stab},
           with different schemes, decreasing $\lambda$ and varying number of
           time steps $m$. 
            {\color{black}The common legend is displayed in the bottom right
            plot. Missing marks mean that the error is larger than
            $10^{-1}$, i.e., the method fails to be unconditionally
            stable.}}
  \label{fig:ad1dex0_diffdeltab}
\end{figure}

\subsection{Nonlinear diffusion-reaction}
  We now turn our attention to the following one-dimensional nonlinear
  diffusion-reaction equation
  \begin{equation}\label{eq:adnonlinear}
    \left\{
    \begin{aligned}
      \partial_t u(t,x)&=\partial_x(a(x)\partial_x u(t,x))+r(u(t,x)),&x\in(-\pi,\pi),\ t\in[0,T],\\
      u(0,x)&=\sin x,
    \end{aligned}\right.
  \end{equation}
  in an inhomogeneous medium with periodic boundary conditions, see
  reference~\cite{PX91}.
  Here we select $a(x)=1+10\sin^2 x$ as {\color{black}the} space
  dependent diffusion coefficient, and the reaction is of quadratic type
  $r(u)= u(1-u)$.
  Similarly to the previous example, we rewrite equation~\eqref{eq:adnonlinear}
  as
\begin{equation}\label{eq:ex1stab}
  \partial_tu(t,x)=\underbrace{\lambda a^{\mathrm{max}} \partial_{xx} u(t,x)}_{A u(t,x)}+
  \underbrace{\partial_x\left((a(x)-\lambda a^{\mathrm{max}})\partial_{x}u(t,x)\right)
  +r(u(t,x))}_{g(x,u(t,x))}.
\end{equation}
  We first discretize in space with a Fourier spectral
  method employing $N=2^{10}$ modes. Then, as for the previous example, the
  temporal schemes can be applied straightforwardly with pointwise operations
  on Fourier coefficients. We simulate until the final time $T=1/10$ with a
  number of time steps equal to $m=2^{12}$ for all the methods. We choose
  different values
  of $\lambda$ and we measure
  the relative errors in the infinity norm with respect to a reference solution
  computed with \textsc{erk2p2}
  as time integrator (applied to original equation~\eqref{eq:adnonlinear}
  semidiscretized in space with spectral finite differences, and with a sufficiently
  large number of time steps). We collect the results in
  Figure~\ref{fig:ad1dpx91_diffdelta}.
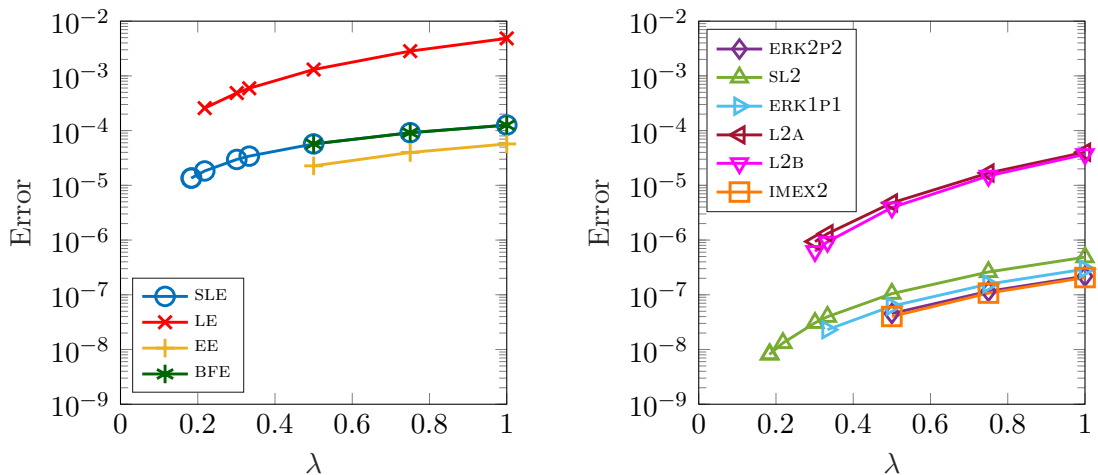
\begin{figure}[htb!]
  \centering
%
%
\definecolor{mycolor1}{rgb}{0.00000,0.44700,0.74100}%
\definecolor{mycolor2}{rgb}{255,0,0}%
\definecolor{mycolor3}{rgb}{0.92900,0.69400,0.12500}%
\definecolor{mycolor4}{rgb}{0.49400,0.18400,0.55600}%
\definecolor{mycolor5}{rgb}{0.46600,0.67400,0.18800}%
\definecolor{mycolor6}{rgb}{0.30100,0.74500,0.93300}%
\definecolor{mycolor7}{rgb}{0.63500,0.07800,0.18400}%
\definecolor{mycolor8}{rgb}{255,0,205}%
\definecolor{mycolor9}{RGB}{0,101,0}%
\definecolor{mycolor10}{RGB}{255,120,0}%
%
\begin{tikzpicture}

\begin{axis}[%
width=2.0in,
height=2.0in,
at={(1.006in,0.722in)},
scale only axis,
xmin=0,
xmax=1,
xlabel style={font=\color{white!15!black}},
xlabel={$\lambda$},
ymode=log,
ymin=1e-09,
ymax=0.01,
yminorticks=true,
ylabel style={font=\color{white!15!black}},
ylabel={Error},
axis background/.style={fill=white},
title style={font=\bfseries},
legend style={at={(0.03,0.03)}, anchor=south west, font=\scriptsize, legend cell align=left, align=left, draw=white!15!black}
]
\addplot [color=mycolor1, line width=1.1pt, mark size=3.5pt, mark=o, mark options={solid, mycolor1}]
  table[row sep=crcr]{%
1	0.000126214156559945\\
0.75	9.159432706208e-05\\
0.5	5.70882140624806e-05\\
0.333333333333333	3.41471469683753e-05\\
0.301	2.97032488312186e-05\\
0.218	1.83072276953998e-05\\
0.183939720585721	1.3637394950495e-05\\
0.183	1.35086205056413e-05\\
};
\addlegendentry{\textsc{sle}}

\addplot [color=mycolor2, line width=1.1pt, mark size=3.5pt, mark=x, mark options={solid, mycolor2}]
  table[row sep=crcr]{%
1	0.00484133061045172\\
0.75	0.00283241622777376\\
0.5	0.00130470469813112\\
0.333333333333333	0.000593226530434661\\
0.301	0.000485857165343979\\
0.218	0.000257438011667187\\
};
\addlegendentry{\textsc{le}}

\addplot [color=mycolor3, line width=1.1pt, mark size=3.5pt, mark=+, mark options={solid, mycolor3}]
  table[row sep=crcr]{%
1	5.7047227927228e-05\\
0.75	3.98546926296974e-05\\
0.5	2.26893194992049e-05\\
};
\addlegendentry{\textsc{ee}}

\addplot [color=mycolor9, line width=1.1pt, mark size=3.5pt, mark=asterisk, mark options={solid, mycolor9}]
  table[row sep=crcr]{%
1	0.000125721579024993\\
0.75	9.13179792691407e-05\\
0.5	5.6965716621269e-05\\
};
\addlegendentry{\textsc{bfe}}

\end{axis}

\begin{axis}[%
width=2.0in,
height=2.0in,
at={(4.0in,0.722in)},
scale only axis,
xmin=0,
xmax=1,
xlabel style={font=\color{white!15!black}},
xlabel={$\lambda$},
ymode=log,
ymin=1e-09,
ymax=0.01,
yminorticks=true,
ylabel style={font=\color{white!15!black}},
ylabel={Error},
axis background/.style={fill=white},
title style={font=\bfseries},
legend style={at={(0.02,0.98)}, anchor=north west, font=\scriptsize, legend cell align=left, align=left, draw=white!15!black}
]
\addplot [color=mycolor4, line width=1.1pt, mark size=3.5pt, mark=diamond, mark options={solid, mycolor4}]
  table[row sep=crcr]{%
1	2.18171665233022e-07\\
0.75	1.15636162487638e-07\\
0.5	4.54794073015633e-08\\
};
\addlegendentry{\textsc{erk2p2}}

\addplot [color=mycolor5, line width=1.1pt, mark size=3.5pt, mark=triangle, mark options={solid, mycolor5}]
  table[row sep=crcr]{%
1	4.86274401250389e-07\\
0.75	2.60852884000447e-07\\
0.5	1.04857304492054e-07\\
0.333333333333334	3.97504337269677e-08\\
0.301	3.07335778727722e-08\\
0.218	1.31520008095652e-08\\
0.183939720585721	8.22425080696326e-09\\
0.183	8.10631356900742e-09\\
};
\addlegendentry{\textsc{sl2}}

\addplot [color=mycolor6, line width=1.1pt, mark size=3.5pt, mark=triangle, mark options={solid, rotate=270, mycolor6}]
  table[row sep=crcr]{%
1	2.94511205133239e-07\\
0.75	1.56755711549669e-07\\
0.5	6.20916462320093e-08\\
0.333333333333333	2.30509896480954e-08\\
};
\addlegendentry{\textsc{erk1p1}}

\addplot [color=mycolor7, line width=1.1pt, mark size=3.5pt, mark=triangle, mark options={solid, rotate=90, mycolor7}]
  table[row sep=crcr]{%
1	4.02776923070653e-05\\
0.75	1.68029828516536e-05\\
0.5	4.78117833428809e-06\\
0.333333333333333	1.30993082623905e-06\\
0.301	9.37740739094311e-07\\
};
\addlegendentry{\textsc{l2a}}

\addplot [color=mycolor8, line width=1.1pt, mark size=3.5pt, mark=triangle, mark options={solid, rotate=180, mycolor8}]
  table[row sep=crcr]{%
1	3.69142379078708e-05\\
0.75	1.48927231763901e-05\\
0.5	3.92512093626621e-06\\
0.333333333333333	9.27428229656738e-07\\
0.301	6.25505479242713e-07\\
};
\addlegendentry{\textsc{l2b}}

\addplot [color=mycolor10, line width=1.1pt, mark size=3.5pt, mark=square, mark options={solid, mycolor10}]
  table[row sep=crcr]{%
1	2.06475713117735e-07\\
0.75	1.07308708086794e-07\\
0.5	4.01678694871665e-08\\
};
\addlegendentry{\textsc{imex2}}

\end{axis}

\end{tikzpicture}%
  \caption{Solution of problem~\eqref{eq:adnonlinear} rewritten as
           equation~\eqref{eq:ex1stab},
           with different schemes of first order (left), of second order (right),
           and varying value $\lambda$. The number of
           time steps
           is fixed to $m=2^{12}$ for each method. Missing marks mean that
           the error is {\color{black}larger than} $10^{-2}$.}
  \label{fig:ad1dpx91_diffdelta}
\end{figure}

Also in this nonlinear case, we observe that the linear analysis predicts
very sharply the amount of diffusion that we can consider in the operator
$A$ while keeping unconditional stability.
Moreover, it is also clear that each method becomes more precise as the value
$\lambda$ decreases, with in general a greater gain for second order methods
than the first order ones.
For completeness, we added also the results obtained with the IMEX schemes
mentioned in Remark~\ref{rem:imex}. Overall, in terms of achieved accuracy,
we observe that the stabilized
Lawson--Euler and the stabilized Lawson2 schemes perform best
among the methods of first and second order, respectively.

\section{Performance comparison}\label{sec:numexpstab}
In this section, we present performance results on two- and three-dimensional
{\color{black}semilinear} advec\-tion-diffusion-reaction equations. All the experiments have been performed
on an Intel Core i7-10750H CPU with six physical cores and 16GB of RAM,
using Matlab R2022a. Moreover, the errors showed in the plots are always
computed in the infinity norm at the final time relatively to a reference
solution computed with the Lawson2b integrator and sufficiently small time
step size.

For comparison, we perform a
  ``standard''
integration of the problem, i.e., we employ the time integrators summarized
in the appendix putting the entire diffusion and advection operators in the
linear part. The remainder $g$ is then equal to the reaction.
{\color{black}For computing the required actions of $\varphi$ functions,
we use} the general purpose \texttt{kiops} method~\cite{GRT18}, whose
Matlab implementation is available
at \url{https://gitlab.com/stephane.gaudreault/kiops/-/tree/master/}.
While there are many techniques available to compute the {\color{black}relevant} matrix
functions~\cite{al2011computing,CCZ20,CCZ22,LPR19,NW12},
\texttt{kiops}
is generally recognized to be among those {\color{black}that} perform best.
This routine requires an input tolerance, which we set as $\tau^{p+1}/100$,
where $\tau$ is the time step size and $p$ is the order of the time
integrator. For the IMEX schemes the linear systems are solved with the
biconjugate gradient stabilized iterative method
(implemented in the internal Matlab function \texttt{bicgstab}). Also for this
routine we set the input tolerance as $\tau^{p+1}/100$.

\subsection{Two-dimensional advection-diffusion-reaction}\label{sec:dr2d}
We start by considering the following two-dimensional
advection-diffusion-reaction
equation
  \begin{equation}\label{eq:ad2d}
    \left\{
    \begin{aligned}
      \partial_t u(t,x_1,x_2)&=\nabla\cdot(a(x_1,x_2)\nabla u(t,x_1,x_2))+
      \nabla\cdot\left(\bb b(x_1,x_2)u(t,x_1,x_2)\right)+r(u(t,x_1,x_2)),\\
      u(0,x_1,x_2)&=\exp(-(x_1^2+x_2^2))
    \end{aligned}\right.
  \end{equation}
  in an inhomogeneous and anisotropic medium with periodic boundary conditions.
  Here $(x_1,x_2)\in(-3\pi,3\pi)^2$, $t\in[0,T]$, the diffusion
  tensor is given by
\begin{equation*}
  a(x_1,x_2) =
  \begin{pmatrix}
    a_{11}(x_1,x_2) & 0 \\
    0 & a_{22}(x_1,x_2)
  \end{pmatrix}
  =
  \begin{pmatrix}
    \frac{1}{2}+\frac{1}{6}\sin^2(x_1)\sin^2(x_2) & 0 \\
    0 & \frac{1}{2}+\frac{1}{6}\cos^2(x_1)\cos^2(x_2)
  \end{pmatrix},
\end{equation*}
the velocity field is given by
\begin{equation*}
  b_1(x_1,x_2)=\frac{1}{5}\sin^2(x_1),\quad
  b_2(x_1,x_2)=\frac{1}{5}\sin^2(x_2),
\end{equation*}
  and the reaction is of quadratic type
  $r(u)= \frac{1}{4}u(1-u)$.
  For the proposed approach we then rewrite equation~\eqref{eq:ad2d} as
\begin{equation}\label{eq:ad2dstab}
    \partial_tu(t,x_1,x_2)=A u(t,x_1,x_2)+g(x_1,x_2,u(t,x_1,x_2)),
\end{equation}
see formulas~\eqref{eq:fullstab}.
The structure of the equation allows for an effective discretization
in space using a Fourier pseudospectral approach. In particular, we consider
$N=N_{x_1}=N_{x_2}$ Fourier modes per direction, and we employ the internal
Matlab functions \texttt{fft2} and
\texttt{ifft2} to perform the direct and inverse transforms, respectively, which
are based on the highly efficient FFTW algorithm~\cite{FJ98}. Note that,
once in Fourier space, the application of the relevant matrix functions is
simply performed by pointwise operations on the Fourier coefficients.

In order to determine the free parameter $\lambda$ in the operator $A$, we
follow the approach described in section~\ref{sec:methods}. In fact, we take
$N=2^6$ Fourier modes,  $m=2^8$ time steps and as final simulation time $T=4$.
The results are presented in Figure~\ref{fig:adr2d_diffl}. For comparison,
we plot also the results obtained with $N=2^8$. As we can observe the two plots
are very similar, but the upper ones have been obtained with a negligible
computational time, still guaranteeing a very good guess for the
optimal $\lambda$.
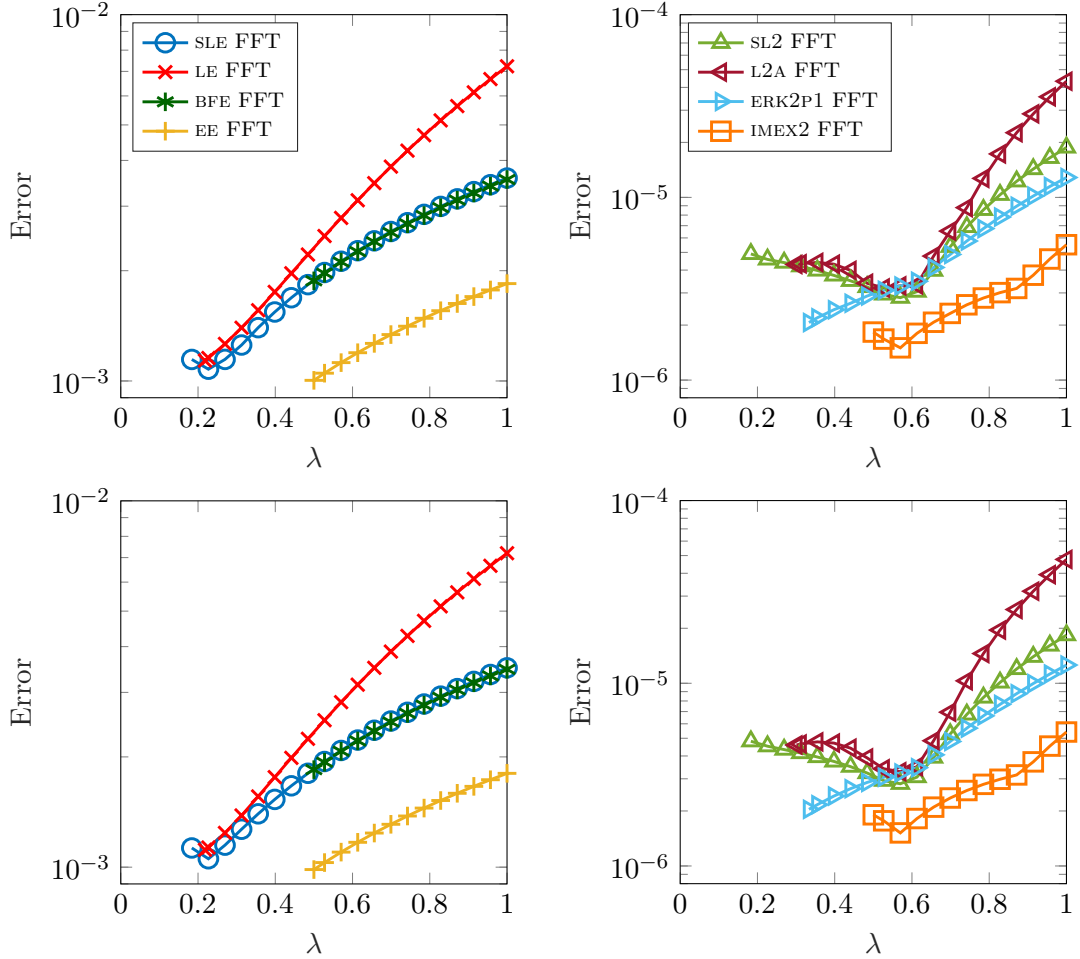
\begin{figure}[htb!]
  \centering
%
%
\definecolor{mycolor1}{rgb}{0.00000,0.44700,0.74100}%
\definecolor{mycolor2}{rgb}{255,0,0}%
\definecolor{mycolor3}{rgb}{0.92900,0.69400,0.12500}%
\definecolor{mycolor9}{RGB}{0,101,0}%
\begin{tikzpicture}

\begin{axis}[%
width=2.0in,
height=2.0in,
at={(0.959in,0.568in)},
scale only axis,
xmin=0,
xmax=1,
xminorticks=true,
xlabel style={font=\color{white!15!black}},
xlabel={$\lambda$},
ymode=log,
ymin=9e-4,
ymax=1e-2,
yminorticks=true,
ylabel style={font=\color{white!15!black}},
ylabel={Error},
axis background/.style={fill=white},
legend style={at={(0.03,0.98)}, anchor=north west,legend cell align=left, font=\scriptsize, align=left, draw=white!15!black}
]
\addplot [color=mycolor1, line width=1.1pt, mark size=3.5pt, mark=o, mark options={solid, mycolor1}]
  table[row sep=crcr]{%
0.183939720585721	0.00114486438980511\\
0.226890261607525	0.00107674912765501\\
0.269840802629329	0.00114516470036293\\
0.312791343651134	0.00125357379525445\\
0.355741884672938	0.00139808089270352\\
0.398692425694742	0.00154271101345854\\
0.441642966716546	0.00168746429196891\\
0.48459350773835	0.00183234086217247\\
0.527544048760154	0.00197734085748588\\
0.570494589781958	0.00212246441079629\\
0.613445130803763	0.0022677116544502\\
0.656395671825567	0.00241308272024512\\
0.699346212847371	0.00255857773941945\\
0.742296753869175	0.00270419684264255\\
0.785247294890979	0.00284994016000517\\
0.828197835912783	0.0029958078210089\\
0.871148376934588	0.00314179995455658\\
0.914098917956392	0.00328791668894168\\
0.957049458978196	0.00343415815183805\\
1	0.00358052447028984\\
};
\addlegendentry{\textsc{sle} FFT}

\addplot [color=mycolor2, line width=1.1pt, mark size=3.5pt, mark=x, mark options={solid, mycolor2}]
  table[row sep=crcr]{%
0.218	0.00113663774013074\\
0.226890261607525	0.00115531275537997\\
0.269840802629329	0.00126239234515102\\
0.312791343651134	0.00139736015458367\\
0.355741884672938	0.00156013967542219\\
0.398692425694742	0.00175063406514963\\
0.441642966716546	0.00196872640655702\\
0.48459350773835	0.00221428003238274\\
0.527544048760154	0.00248713891196803\\
0.570494589781958	0.00278712809690544\\
0.613445130803763	0.00311405422175105\\
0.656395671825567	0.00346770605584345\\
0.699346212847371	0.00384785510145182\\
0.742296753869175	0.00425425623368125\\
0.785247294890979	0.00468664837724583\\
0.828197835912783	0.00514475521483576\\
0.871148376934588	0.00562828592243012\\
0.914098917956392	0.00613693592628579\\
0.957049458978196	0.00667038767683849\\
1	0.00722831143487964\\
};
\addlegendentry{\textsc{le} FFT}

\addplot [color=mycolor9, line width=1.1pt, mark size=3.5pt, mark=asterisk, mark options={solid, mycolor9}]
  table[row sep=crcr]{%
0.5	0.00187650492353341\\
0.527544048760154	0.00196854123871222\\
0.570494589781958	0.00211205389784098\\
0.613445130803763	0.00225556247461166\\
0.656395671825567	0.00239906684660721\\
0.699346212847371	0.00254256689086858\\
0.742296753869175	0.00268606248395073\\
0.785247294890979	0.00282955350192462\\
0.828197835912783	0.00297303982031857\\
0.871148376934588	0.00311652131418391\\
0.914098917956392	0.00325999785809258\\
0.957049458978196	0.00340346932608259\\
1	0.00354693559170931\\
};
\addlegendentry{\textsc{bfe} FFT}

\addplot [color=mycolor3, line width=1.1pt, mark size=3.5pt, mark=+, mark options={solid, mycolor3}]
  table[row sep=crcr]{%
0.5	0.00100392990270124\\
0.527544048760154	0.00105014924706303\\
0.570494589781958	0.00112223738121445\\
0.613445130803763	0.00119434555569133\\
0.656395671825567	0.00126647375809831\\
0.699346212847371	0.00133862197594305\\
0.742296753869175	0.00141079019663645\\
0.785247294890979	0.00148297840749096\\
0.828197835912783	0.0015551865957223\\
0.871148376934588	0.00162741474844742\\
0.914098917956392	0.00169966285268491\\
0.957049458978196	0.00177193089535543\\
1	0.00184421886327937\\
};
\addlegendentry{\textsc{ee} FFT}

\end{axis}
\end{tikzpicture}%
\hspace{10pt}
%
%
\definecolor{mycolor4}{rgb}{0.49400,0.18400,0.55600}%
\definecolor{mycolor5}{rgb}{0.46600,0.67400,0.18800}%
\definecolor{mycolor6}{rgb}{0.30100,0.74500,0.93300}%
\definecolor{mycolor7}{rgb}{0.63500,0.07800,0.18400}%
\definecolor{mycolor8}{rgb}{255,0,205}%
\definecolor{mycolor10}{RGB}{255,120,0}%
\begin{tikzpicture}

\begin{axis}[%
width=2.0in,
height=2.0in,
at={(0.959in,0.568in)},
scale only axis,
xmin=0,
xmax=1,
xminorticks=true,
xlabel style={font=\color{white!15!black}},
xlabel={$\lambda$},
ymode=log,
ymin=8e-7,
ymax=1e-4,
yminorticks=true,
ylabel style={font=\color{white!15!black}},
ylabel={Error},
axis background/.style={fill=white},
legend style={at={(0.02,0.98)}, anchor=north west,legend cell align=left,
font=\scriptsize, align=left, draw=white!15!black}
]
\addplot [color=mycolor5, line width=1.1pt, mark size=3.5pt, mark=triangle, mark options={solid, mycolor5}]
  table[row sep=crcr]{%
0.183	4.90464201995568e-06\\
0.226	4.5758730815384e-06\\
0.269	4.39754422438348e-06\\
0.312	4.19989333183979e-06\\
0.355	3.98088777131286e-06\\
0.398	3.74072027576217e-06\\
0.441	3.49409448267367e-06\\
0.484	3.23341441850625e-06\\
0.527	2.94345320788361e-06\\
0.57	2.82577128864082e-06\\
0.613	3.0499145682129e-06\\
0.656	3.97205753945719e-06\\
0.699	5.36724562712537e-06\\
0.742	6.89305757205227e-06\\
0.785	8.54901590005244e-06\\
0.828	1.03346493655512e-05\\
0.871	1.22494928081435e-05\\
0.914	1.4293087014754e-05\\
0.957	1.646497858612e-05\\
1	1.87647198112497e-05\\
};
\addlegendentry{\textsc{sl2} FFT}

\addplot [color=mycolor7, line width=1.1pt, mark size=3.5pt, mark=triangle, mark options={solid, rotate=90, mycolor7}]
  table[row sep=crcr]{%
0.301	4.29271764766399e-06\\
0.312	4.33308848474681e-06\\
0.355	4.40729610835527e-06\\
0.398	4.30908123684305e-06\\
0.441	3.98665073222241e-06\\
0.484	3.3887396191289e-06\\
0.527	3.19294358246004e-06\\
0.57	3.28055231234961e-06\\
0.613	3.32098590032413e-06\\
0.656	4.77211547600556e-06\\
0.699	6.53809821587397e-06\\
0.742	8.78957737727705e-06\\
0.785	1.27080398574182e-05\\
0.828	1.72936619651855e-05\\
0.871	2.25931258647271e-05\\
0.914	2.86526054408546e-05\\
0.957	3.55177638109832e-05\\
1	4.32337501668689e-05\\
};
\addlegendentry{\textsc{l2a} FFT}

\addplot [color=mycolor6, line width=1.1pt, mark size=3.5pt, mark=triangle, mark options={solid, rotate=270, mycolor6}]
  table[row sep=crcr]{%
0.333333333333333	2.06542287246143e-06\\
0.355	2.1940252287949e-06\\
0.398	2.42791166677976e-06\\
0.441	2.63346075759849e-06\\
0.484	2.88556176979526e-06\\
0.527	3.10095213498975e-06\\
0.57	3.27472906815074e-06\\
0.613	3.47512722494846e-06\\
0.656	4.13621710469769e-06\\
0.699	4.87243332181897e-06\\
0.742	5.76932533379075e-06\\
0.785	6.74889319305577e-06\\
0.828	7.80902173062024e-06\\
0.871	8.94941991801669e-06\\
0.914	1.01697997364854e-05\\
0.957	1.1469876111525e-05\\
1	1.2849366856088e-05\\
};
\addlegendentry{\textsc{erk2p1} FFT}

\addplot [color=mycolor10, line width=1.1pt, mark size=3.5pt, mark=square, mark options={solid, mycolor10}]
  table[row sep=crcr]{%
0.5	1.8320587235607e-06\\
0.527	1.67682734796389e-06\\
0.57	1.50152850905626e-06\\
0.613	1.80302037189122e-06\\
0.656	2.07311306080361e-06\\
0.699	2.31724414025097e-06\\
0.742	2.58240829302981e-06\\
0.785	2.81363284677798e-06\\
0.828	3.0109092156029e-06\\
0.871	3.17423000729107e-06\\
0.914	3.75131801876945e-06\\
0.957	4.59975408944807e-06\\
1	5.50670410903476e-06\\
};
\addlegendentry{\textsc{imex2} FFT}
\end{axis}

\end{tikzpicture}%
 \\
%
%
\definecolor{mycolor1}{rgb}{0.00000,0.44700,0.74100}%
\definecolor{mycolor2}{rgb}{255,0,0}%
\definecolor{mycolor3}{rgb}{0.92900,0.69400,0.12500}%
\definecolor{mycolor9}{RGB}{0,101,0}%
\begin{tikzpicture}

\begin{axis}[%
width=2.0in,
height=2.0in,
at={(0.959in,0.568in)},
scale only axis,
xmin=0,
xmax=1,
xminorticks=true,
xlabel style={font=\color{white!15!black}},
xlabel={$\lambda$},
ymode=log,
ymin=9e-4,
ymax=1e-2,
yminorticks=true,
ylabel style={font=\color{white!15!black}},
ylabel={Error},
axis background/.style={fill=white},
]
\addplot [color=mycolor1, line width=1.1pt, mark size=3.5pt, mark=o, mark options={solid, mycolor1}]
  table[row sep=crcr]{%
0.183939720585721	0.00112839318670118\\
0.226890261607525	0.00105411665753365\\
0.269840802629329	0.00114849476103801\\
0.312791343651134	0.00126918396500743\\
0.355741884672938	0.00139835980336247\\
0.398692425694742	0.00153116196423713\\
0.441642966716546	0.00166766586466882\\
0.48459350773835	0.00180441218951746\\
0.527544048760154	0.00194320122901103\\
0.570494589781958	0.00208240073194602\\
0.613445130803763	0.00222342011028595\\
0.656395671825567	0.00236456472768175\\
0.699346212847371	0.00250583473537725\\
0.742296753869175	0.0026472302841781\\
0.785247294890979	0.00278875152444015\\
0.828197835912783	0.00293039860606283\\
0.871148376934588	0.00307217167848141\\
0.914098917956392	0.0032147169595043\\
0.957049458978196	0.0033577026592477\\
1	0.00350081042052462\\
};

\addplot [color=mycolor2, line width=1.1pt, mark size=3.5pt, mark=x, mark options={solid, mycolor2}]
  table[row sep=crcr]{%
0.218	0.00111519422971395\\
0.226890261607525	0.00113156457386746\\
0.269840802629329	0.0012402215653698\\
0.312791343651134	0.00138419696265212\\
0.355741884672938	0.00156000144111313\\
0.398692425694742	0.00176149409046911\\
0.441642966716546	0.00198854616303711\\
0.48459350773835	0.00224101021243033\\
0.527544048760154	0.00251872055226816\\
0.570494589781958	0.00282149377290749\\
0.613445130803763	0.00314912931099792\\
0.656395671825567	0.00350141006683204\\
0.699346212847371	0.00387810306364554\\
0.742296753869175	0.0042789601433881\\
0.785247294890979	0.00470371869335202\\
0.828197835912783	0.00515210239770883\\
0.871148376934588	0.00562637132998349\\
0.914098917956392	0.00612622777917442\\
0.957049458978196	0.00664949860968541\\
1	0.00719585866877569\\
};

\addplot [color=mycolor9, line width=1.1pt, mark size=3.5pt, mark=asterisk, mark options={solid, mycolor9}]
  table[row sep=crcr]{%
0.5	0.00184615334592339\\
0.527544048760154	0.001934301893614\\
0.570494589781958	0.00207209459880053\\
0.613445130803763	0.00221141069953581\\
0.656395671825567	0.00235072769288878\\
0.699346212847371	0.00249004546876016\\
0.742296753869175	0.00262936391652126\\
0.785247294890979	0.00276868292503911\\
0.828197835912783	0.00290800238260274\\
0.871148376934588	0.00304732217698832\\
0.914098917956392	0.00318742277073187\\
0.957049458978196	0.00332770045520846\\
1	0.00346797305898219\\
};

\addplot [color=mycolor3, line width=1.1pt, mark size=3.5pt, mark=+, mark options={solid, mycolor3}]
  table[row sep=crcr]{%
0.5	0.00098508490644841\\
0.527544048760154	0.00102933486127439\\
0.570494589781958	0.00109923912483275\\
0.613445130803763	0.00116922548763291\\
0.656395671825567	0.00123923251968875\\
0.699346212847371	0.00130926021171608\\
0.742296753869175	0.00137937446139251\\
0.785247294890979	0.00144995564357365\\
0.828197835912783	0.0015205563579797\\
0.871148376934588	0.00159117659201857\\
0.914098917956392	0.00166181633300185\\
0.957049458978196	0.0017326957873409\\
1	0.00180374379738495\\
};

\end{axis}
\end{tikzpicture}%
\hspace{10pt}
%
%
\definecolor{mycolor4}{rgb}{0.49400,0.18400,0.55600}%
\definecolor{mycolor5}{rgb}{0.46600,0.67400,0.18800}%
\definecolor{mycolor6}{rgb}{0.30100,0.74500,0.93300}%
\definecolor{mycolor7}{rgb}{0.63500,0.07800,0.18400}%
\definecolor{mycolor8}{rgb}{255,0,205}%
\definecolor{mycolor10}{RGB}{255,120,0}%
\begin{tikzpicture}

\begin{axis}[%
width=2.0in,
height=2.0in,
at={(0.959in,0.568in)},
scale only axis,
xmin=0,
xmax=1,
xminorticks=true,
xlabel style={font=\color{white!15!black}},
xlabel={$\lambda$},
ymode=log,
ymin=8e-7,
ymax=1e-4,
yminorticks=true,
ylabel style={font=\color{white!15!black}},
ylabel={Error},
axis background/.style={fill=white},
legend style={at={(0.03,0.98)}, anchor=north west,legend cell align=left, font=\footnotesize, align=left, draw=white!15!black}
]
\addplot [color=mycolor5, line width=1.1pt, mark size=3.5pt, mark=triangle, mark options={solid, mycolor5}]
  table[row sep=crcr]{%
0.183	4.81656348160345e-06\\
0.226	4.52454387152881e-06\\
0.269	4.33574295483652e-06\\
0.312	4.15111594046095e-06\\
0.355	3.94902210784038e-06\\
0.398	3.73429577815255e-06\\
0.441	3.49726335326386e-06\\
0.484	3.22677432364227e-06\\
0.527	2.92926658194726e-06\\
0.57	2.81805321597858e-06\\
0.613	3.0663666050553e-06\\
0.656	3.92125695436969e-06\\
0.699	5.27395639777593e-06\\
0.742	6.75368701082896e-06\\
0.785	8.36008587873668e-06\\
0.828	1.01047973484615e-05\\
0.871	1.19769326407462e-05\\
0.914	1.39749494065046e-05\\
0.957	1.60984048629519e-05\\
1	1.83468617429099e-05\\
};

\addplot [color=mycolor7, line width=1.1pt, mark size=3.5pt, mark=triangle, mark options={solid, rotate=90, mycolor7}]
  table[row sep=crcr]{%
0.301	4.59335818868854e-06\\
0.312	4.65136300951682e-06\\
0.355	4.77865397667558e-06\\
0.398	4.7064217437193e-06\\
0.441	4.41418396666878e-06\\
0.484	3.90395704886918e-06\\
0.527	3.40460882210155e-06\\
0.57	3.28109154424903e-06\\
0.613	3.4929777797473e-06\\
0.656	4.84632759843286e-06\\
0.699	6.94548876163771e-06\\
0.742	1.0310677437812e-05\\
0.785	1.45446319511352e-05\\
0.828	1.95621674436522e-05\\
0.871	2.53342200325501e-05\\
0.914	3.19064581639744e-05\\
0.957	3.93238948707502e-05\\
1	4.76308840101909e-05\\
};

\addplot [color=mycolor6, line width=1.1pt, mark size=3.5pt, mark=triangle, mark options={solid, rotate=270, mycolor6}]
  table[row sep=crcr]{%
0.333333333333333	2.04162061371193e-06\\
0.355	2.16637674108794e-06\\
0.398	2.41353294062124e-06\\
0.441	2.64937548935514e-06\\
0.484	2.86583547541403e-06\\
0.527	3.05949889959941e-06\\
0.57	3.23084441200319e-06\\
0.613	3.45163395404974e-06\\
0.656	4.05920790712521e-06\\
0.699	4.80242181438907e-06\\
0.742	5.6787130848679e-06\\
0.785	6.63287094609197e-06\\
0.828	7.66461971236856e-06\\
0.871	8.77368638780147e-06\\
0.914	9.96606199417702e-06\\
0.957	1.12369499877087e-05\\
1	1.25851667577985e-05\\
};

\addplot [color=mycolor10, line width=1.1pt, mark size=3.5pt, mark=square, mark options={solid, mycolor10}]
  table[row sep=crcr]{%
0.5	1.90598154429683e-06\\
0.527	1.7666963187653e-06\\
0.57	1.5172472774463e-06\\
0.613	1.81802898725097e-06\\
0.656	2.09900813409153e-06\\
0.699	2.35472929438783e-06\\
0.742	2.58794436571756e-06\\
0.785	2.79780644277079e-06\\
0.828	2.98291833279462e-06\\
0.871	3.14997068991754e-06\\
0.914	3.71400235458216e-06\\
0.957	4.53120657190448e-06\\
1	5.41036538002148e-06\\
};

\end{axis}

\end{tikzpicture}%
  \caption{Solution of problem~\eqref{eq:ad2d} rewritten as
           equation~\eqref{eq:ad2dstab},
           with different first order schemes (left) and second order schemes
           (right), for varying value $\lambda$.
           The number of time steps
           is fixed to $m=2^8$ for each method, while the number of degrees of
           freedom for each spatial variable is $N=2^6$ (top) and $N=2^8$
           (bottom). }
  \label{fig:adr2d_diffl}
\end{figure}

A summary of the values of $\lambda$ obtained in this way for all the
integrators
under consideration is given in Table~\ref{tab:adr2d_lambdaopt}. Remark that,
in this example, it is not always the case that the lower bounds stemming from
the linear stability analysis correspond to the lowest errors, both for the
exponential and for the IMEX schemes (cf. Table~\ref{tab:stab}).
\begin{table}[htb!]
  \centering
  \begin{tabular}{c||cccc|cccc}
    method & \textsc{bfe} & \textsc{ee} & \textsc{le} & \textsc{sle} &
    \textsc{erk2p1} & \textsc{imex2} & \textsc{l2a} & \textsc{sl2}\\
\hline
  \rule{0pt}{12pt} value of $\lambda$ & 0.50 & 0.50 & 0.22 & 0.23 & 0.33 & 0.57 & 0.53 & 0.57
  \end{tabular}
  \caption{Values of $\lambda$,
    for the different schemes that are employed to integrate
    problem~\eqref{eq:ad2d} rewritten as equation~\eqref{eq:ad2dstab}
    (see also Figure~\ref{fig:adr2d_diffl}).}
  \label{tab:adr2d_lambdaopt}
\end{table}

Now, employing for each integrator the corresponding value of $\lambda$ found,
we present the actual performance results using $N=2^8$ degrees of freedom for
every spatial variable.
As a comparison, we consider the original equation~\eqref{eq:ad2d} {\color{black}
semidiscretized in space with spectral finite differences} and
solved numerically with the same temporal integrators. {\color{black} In
addition, we also perform the time evolution with the exponential
Rosenbrock--Euler 
scheme (see the appendix), that we label \textsc{erbe}.
This method is of second order, requires the action of a single $\varphi_1$ 
function and employs as linear operator the Jacobian of the right-hand side of the
equation (i.e., it exploits in the linear part also information coming from the 
reaction term, which may be beneficial)}.

The number of time steps is set to $2^{\ell_1}$, with
$\ell_1=9,\ldots,12$, for the first order integrators, while $2^{\ell_2}$, with
$\ell_2=8,\ldots,11$, for the second order methods. The final simulation time
is again $T=4$.
The results are collected in the CPU diagrams of Figure~\ref{fig:cpudiag2d}.
\begin{figure}[htb!]
  \centering
%
%
\definecolor{mycolor1}{rgb}{0.00000,0.44700,0.74100}%
\definecolor{mycolor2}{rgb}{255,0,0}%
\definecolor{mycolor3}{rgb}{0.92900,0.69400,0.12500}%
\definecolor{mycolor4}{rgb}{0.49400,0.18400,0.55600}%
\definecolor{mycolor5}{rgb}{0.46600,0.67400,0.18800}%
\definecolor{mycolor6}{rgb}{0.30100,0.74500,0.93300}%
\definecolor{mycolor7}{rgb}{0.63500,0.07800,0.18400}%
\definecolor{mycolor8}{rgb}{255,0,205}%
\definecolor{mycolor9}{RGB}{0,101,0}%
\definecolor{mycolor10}{RGB}{255,120,0}%
\definecolor{mycolor11}{RGB}{192,192,192}%
\begin{tikzpicture}

\begin{axis}[%
width=2.0in,
height=2.5in,
at={(0.959in,0.568in)},
scale only axis,
xmode=log,
xmin=1,
xmax=150,
xminorticks=true,
xlabel style={font=\color{white!15!black}},
xlabel={Wall-clock time},
ymode=log,
ymin=7e-6,
ymax=2e-3,
yminorticks=true,
ylabel style={font=\color{white!15!black}},
ylabel={Error},
axis background/.style={fill=white},
title style={font=\bfseries},
legend style={at={(0.02,0.02)},anchor=south west,legend cell align=left, font=\scriptsize, align=left, draw=white!15!black, legend columns=2},
]
\addplot [color=mycolor1, dashed, line width=1.1pt, mark size=3.5pt, mark=o, mark options={solid, mycolor1}]
  table[row sep=crcr]{%
12.098047	0.00110086168057716\\
25.139657	0.000549976305138252\\
50.920886	0.000274876443990114\\
101.886563	0.000137412305099492\\
};
\addlegendentry{\textsc{sle}}

\addplot [color=mycolor2, dashed, line width=1.1pt, mark size=3.5pt, mark=x, mark options={solid, mycolor2}]
  table[row sep=crcr]{%
11.853121	0.0007780910152167\\
23.909601	0.000388943182620642\\
48.260165	0.000194444639994115\\
96.214306	9.72142365224352e-05\\
};
\addlegendentry{\textsc{le}}

\addplot [color=mycolor3, dashed, line width=1.1pt, mark size=3.5pt, mark=+, mark options={solid, mycolor3}]
  table[row sep=crcr]{%
12.588906	0.000580006741533052\\
25.395485	0.000289982334491277\\
50.832936	0.000144987862443679\\
102.500733	7.24950787114666e-05\\
};
\addlegendentry{\textsc{ee}}

\addplot [color=mycolor9, dashed, line width=1.1pt, mark size=3.5pt, mark=asterisk, mark options={solid, mycolor9}]
  table[row sep=crcr]{%
6.744604	0.00111663886099136\\
12.459467	0.000550244480534153\\
20.675822	0.000277872362249588\\
34.044228	0.000137968450409394\\
};
\addlegendentry{\textsc{bfe}}

\addplot [color=mycolor1, line width=1.1pt, mark size=3.5pt, mark=o, mark options={solid, mycolor1}]
  table[row sep=crcr]{%
1.839315	0.000526370130694756\\
3.438333	0.00026308016012716\\
6.808527	0.000131515973432218\\
13.624456	6.57539831737838e-05\\
};
\addlegendentry{\textsc{sle} FFT}

\addplot [color=mycolor2, line width=1.1pt, mark size=3.5pt, mark=x, mark options={solid, mycolor2}]
  table[row sep=crcr]{%
1.843468	0.000556992871400477\\
3.532813	0.000278345630522203\\
6.94337500000001	0.000139134015089943\\
13.884556	6.95560412712643e-05\\
};
\addlegendentry{\textsc{le} FFT}

\addplot [color=mycolor3, line width=1.1pt, mark size=3.5pt, mark=+, mark options={solid, mycolor3}]
  table[row sep=crcr]{%
1.84071	0.000492390890695822\\
3.452812	0.000246160186987895\\
6.874808	0.000123073452257836\\
13.717143	6.15371819302634e-05\\
};
\addlegendentry{\textsc{ee} FFT}

\addplot [color=mycolor9, line width=1.1pt, mark size=3.5pt, mark=asterisk, mark options={solid, mycolor9}]
  table[row sep=crcr]{%
1.787308	0.000923782637668484\\
3.199085	0.000462069648677115\\
6.109236	0.000231081480774508\\
12.843855	0.000115554507864472\\
};
\addlegendentry{\textsc{bfe} FFT}

\end{axis}

\end{tikzpicture}%
%
%
\definecolor{mycolor1}{rgb}{0.00000,0.44700,0.74100}%
\definecolor{mycolor2}{rgb}{255,0,0}%
\definecolor{mycolor3}{rgb}{0.92900,0.69400,0.12500}%
\definecolor{mycolor4}{rgb}{0.49400,0.18400,0.55600}%
\definecolor{mycolor5}{rgb}{0.46600,0.67400,0.18800}%
\definecolor{mycolor6}{rgb}{0.30100,0.74500,0.93300}%
\definecolor{mycolor7}{rgb}{0.63500,0.07800,0.18400}%
\definecolor{mycolor8}{rgb}{255,0,205}%
\definecolor{mycolor9}{RGB}{0,101,0}%
\definecolor{mycolor10}{RGB}{255,120,0}%
\begin{tikzpicture}

\begin{axis}[%
width=2.0in,
height=2.5in,
at={(0.959in,0.568in)},
scale only axis,
xmode=log,
xmin=1,
xmax=200,
xminorticks=true,
xlabel style={font=\color{white!15!black}},
xlabel={Wall-clock time},
ymode=log,
ymin=1e-10,
ymax=1e-5,
yminorticks=true,
ylabel style={font=\color{white!15!black}},
ylabel={Error},
axis background/.style={fill=white},
title style={font=\bfseries},
legend style={at={(0.02,0.02)}, anchor=south west, legend cell align=left, font=\scriptsize, align=left, draw=white!15!black, legend columns=2},
]

\addplot [color=mycolor5, dashed, line width=1.1pt, mark size=3.5pt, mark=triangle, mark options={solid, mycolor5}]
  table[row sep=crcr]{%
13.086028	3.63508110156608e-06\\
25.802046	9.09812411399881e-07\\
54.695422	2.29425647535178e-07\\
108.223447	5.94718474565569e-08\\
};
\addlegendentry{\textsc{sl2}}

\addplot [color=mycolor6, dashed, line width=1.1pt, mark size=3.5pt, mark=triangle, mark options={solid, rotate=270, mycolor6}]
  table[row sep=crcr]{%
12.426066	4.90287505299534e-06\\
25.163217	1.22960593005793e-06\\
52.7485770000001	3.12616136595557e-07\\
105.999093	8.33365496316061e-08\\
};
\addlegendentry{\textsc{erk2p1}}

\addplot [color=mycolor7, dashed, line width=1.1pt, mark size=3.5pt, mark=triangle, mark options={solid, rotate=90, mycolor7}]
  table[row sep=crcr]{%
17.901806	1.02634705493567e-06\\
35.861589	2.60042677132835e-07\\
75.99898	6.77391730775318e-08\\
152.264781	1.96648604766944e-08\\
};
\addlegendentry{\textsc{l2a}}

\addplot [color=teal, dashed, line width=1.1pt, mark size=3.5pt, mark=o, mark options={solid, teal}]
  table[row sep=crcr]{%
6.790854	3.92849203169891e-06\\
13.796782	9.82035337450601e-07\\
27.769881	2.47579765820689e-07\\
54.764825	6.45697995289635e-08\\
};
\addlegendentry{\textsc{erbe}}

\addplot [color=mycolor10, dashed, line width=1.1pt, mark size=3.5pt, mark=square, mark options={solid, mycolor10}]
  table[row sep=crcr]{%
9.381463	3.45691899912649e-06\\
16.414386	7.41084501035014e-07\\
26.026193	1.8256742904577e-07\\
40.267477	4.26714632690844e-08\\
};
\addlegendentry{\textsc{imex2}}

\addplot [color=mycolor5, line width=1.1pt, mark size=3.5pt, mark=triangle, mark options={solid, mycolor5}]
  table[row sep=crcr]{%
1.912917	2.81805321538333e-06\\
3.677364	7.05801714811072e-07\\
7.25856	1.78884518051298e-07\\
14.18976	4.73102521671667e-08\\
};
\addlegendentry{\textsc{sl2} FFT}

\addplot [color=mycolor6, line width=1.1pt, mark size=3.5pt, mark=triangle, mark options={solid, rotate=270, mycolor6}]
  table[row sep=crcr]{%
1.765997	2.0226978162428e-06\\
3.513002	4.95089019900098e-07\\
7.03526000000001	1.20353812951466e-07\\
13.426829	2.77984428279036e-08\\
};
\addlegendentry{\textsc{erk2p1} FFT}

\addplot [color=mycolor7, line width=1.1pt, mark size=3.5pt, mark=triangle, mark options={solid, rotate=90, mycolor7}]
  table[row sep=crcr]{%
1.837949	3.38310862834931e-06\\
3.570315	8.49222977290802e-07\\
6.589088	2.15171994950246e-07\\
13.299486	5.63603195337909e-08\\
};
\addlegendentry{\textsc{l2a} FFT}

\addplot [color=mycolor10, line width=1.1pt, mark size=3.5pt, mark=square, mark options={solid, mycolor10}]
  table[row sep=crcr]{%
1.740329	1.51724727792249e-06\\
3.42626	3.75302161868081e-07\\
6.685533	9.08802211725367e-08\\
13.270069	1.99071882398876e-08\\
};
\addlegendentry{\textsc{imex2} FFT}

\end{axis}

\end{tikzpicture}%
  \caption{Results for the simulations of problem~\eqref{eq:ad2d}
  (dashed lines), rewritten as equation~\eqref{eq:ad2dstab}
  (solid lines), with different integrators of first order (left, number of time
  steps $m=2^{\ell_1}$ with $\ell_1=9,\ldots,12$) and of second order (right,
  number of time steps $m=2^{\ell_2}$ with $\ell_2=8,\ldots,11$).}
  \label{fig:cpudiag2d}
\end{figure}
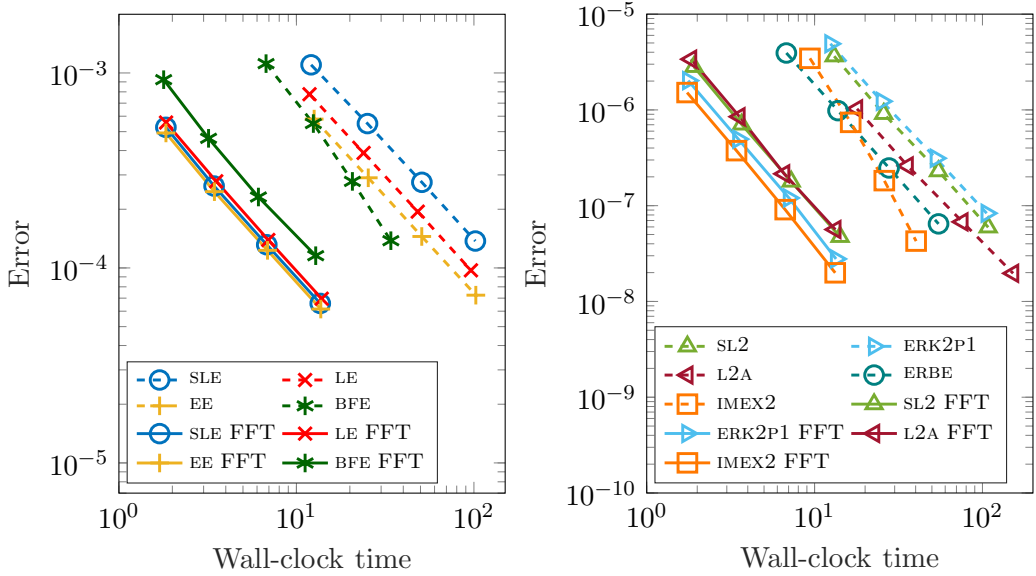
As we can see, overall the proposed approach allows to obtain 
slightly better errors with respect to the results in the original formulation,
{\color{black}with the exception of the Lawson2a scheme.
Indeed, the choice of the parameter $\lambda$ in the 
acceleration technique is determined by stability considerations
only, and not by 
accuracy ones. Hence, a decrease in error is not to be expected, in general}.
The computational
time of the proposed approach is less than the corresponding integration in the
original formulation. Among the first order schemes the exponential
 methods perform equally well, while for
 the second order methods the implicit-explicit2 scheme is the one that
 shows slightly better results. Note also that, when integrating
 equation~\eqref{eq:ad2dstab}
with the proposed approach, the wall-clock time is basically independent of the
chosen time integrators, both for first and second order methods.
This is expected, as the most costly operations in this case are the Fourier
and inverse Fourier transforms. {\color{black}In fact each scheme of order one 
requires 7 transforms per time step, while we need 13 for the second order
methods.
Finally, note that as expected the \textsc{erbe} method requires
less wall-clock time than the \textsc{l2a} and the \textsc{erk2p1} scheme (since
it requires just a single action of $\varphi$ function).
However, overall the computational gain of the proposed approach is still superior.}

\subsection{Three-dimensional advection-diffusion-reaction}\label{sec:adr3d}
We now consider the following three-dimensional
advection-diffusion-reaction equation
  \begin{equation}\label{eq:adr3d}
    \left\{
    \begin{aligned}
      \partial_t u(t,x_1,x_2,x_3)&=a(x_1,x_2,x_3)\Delta u(t,x_1,x_2,x_3)+
      b \sum_{\mu=1}^3 \partial_{x_\mu} u(t,x_1,x_2,x_3) +r(u(t,x_1,x_2,x_3)),\\
      u(0,x_1,x_2,x_3)&=\left(\frac{27}{4}\right)^3x_1x_2x_3(1-x_1)^2(1-x_2)^2(1-x_3)^2,
    \end{aligned}\right.
  \end{equation}
  in the spatial domain $\Omega = (0,1)^3$ and $t\in[0,T]$.
  We equip the problem with homogeneous Dirichlet boundary conditions on the
  set $\{(x_1,x_2,x_3)\in\partial\Omega: x_1x_2x_3=0\}$ and with homogeneous
  Neumann
  boundary conditions elsewhere.
  We set the diffusion coefficient to
  \begin{equation*}
    a(x_1,x_2,x_3) = \frac{1}{10}\ee^{-(x_1-1/2)^2-(x_2-1/2)^2-(x_3-1/2)^2},
  \end{equation*}
  while the reaction is of cubic type $r(u)=u(1+u^2)$.

  For the proposed approach we rewrite equation~\eqref{eq:adr3d} as
\begin{equation}\label{eq:adr3dstab}
\begin{split}
  \partial_t u(t,x_1,x_2,x_3)&=
  \underbrace{(\lambda a^{\mathrm{max}}\Delta +
    b(\partial_{x_1} +\partial_{x_2}+\partial_{x_3}))
    u(t,x_1,x_2,x_3)}_{A u(t,x_1,x_2,x_3)}\\
  &+ \underbrace{(a(x_1,x_2,x_3)-\lambda a^{\mathrm{max}})\Delta u(t,x_1,x_2,x_3) +
    r(u(t,x_1,x_2,x_3))}_{g(x_1,x_2,x_3,u(t,x_1,x_2,x_3))},
\end{split}
\end{equation}
{\color{black}
where
\begin{equation*}
  a^{\mathrm{max}}=\max_{x_1, x_2, x_3}a(x_1,x_2,x_3).
\end{equation*}
}%
We discretize this equation in space  with standard second order centered finite
differences, using $N=N_{x_1}=N_{x_2}=N_{x_3}$ discretization points for each
direction.
By doing so, the linear operator $A$ in equation~\eqref{eq:adr3dstab}
is approximated by a matrix with Kronecker sum structure
$A_3 \oplus A_2 \oplus A_1$,
where $A_\mu\in\RR^{N_{x_\mu} \times N_{x_\mu}}$ is the discretization matrix
of the operator $\lambda a^{\mathrm{max}}\partial_{x_\mu x_\mu}
+ b \partial_{x_\mu}$.
Hence, for those exponential integrators that just require the exponential
function (i.e., the ones of Lawson type), it is possible to employ $\mu$-mode
based techniques in order to efficiently compute the needed actions of the
matrix exponential via Tucker operators 
{\color{black}
exploiting the equivalence
\begin{equation*}
  \ee^{A_3 \oplus A_2 \oplus A_1}\bb v =
  \mathrm{vec}\left(\bb V \times_1 \ee^{A_1} \times_2 \ee^{A_2} \times_3 \ee^{A_3}\right),
  \quad \bb V \in \RR^{N\times N \times N},
  \quad \bb v = \mathrm{vec}(\bb V),
\end{equation*}
see references~\cite{CCEOZ22,CCZ22b} for more details}.

Similarly to the previous example, in order to determine the free parameter
$\lambda$ in $A$, we use the technique described in section~\ref{sec:methods}.
We set $N=2^4$ and perform a simulation
with $m=2^8$ time steps and final simulation time $T=1/4$. The outcome,
for two different choices
of advection parameter ($b=-0.01$ and $b=-1$), is summarized in
Figure~\ref{fig:adr3d_diffl}, and the resulting values of $\lambda$ actually
employed in the experiments are given in
Table~\ref{tab:adr3d_lambdaopt}.
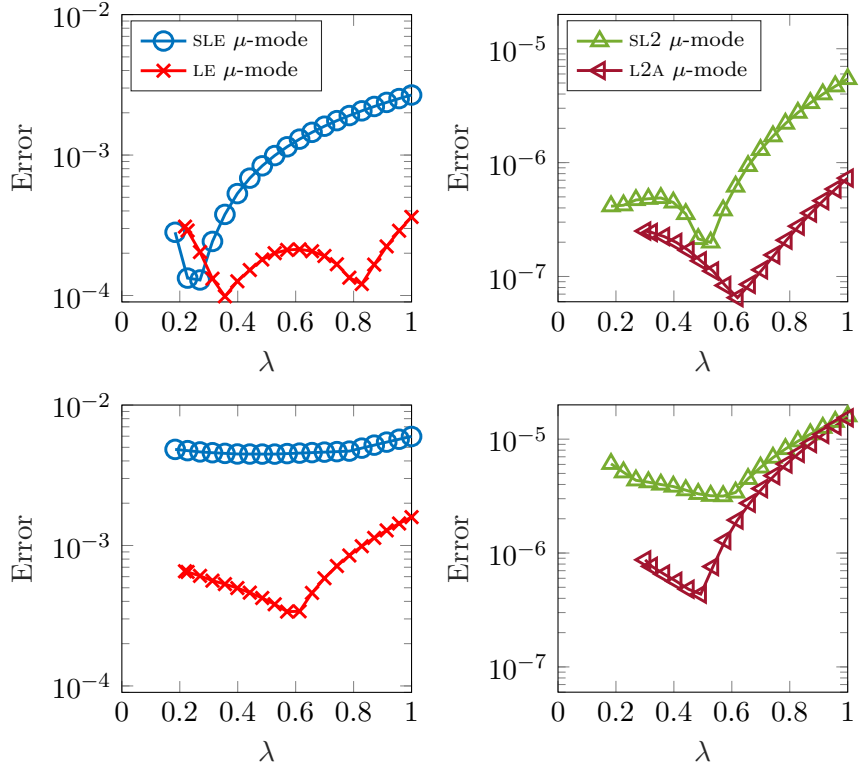
\begin{figure}[htb!]
  \centering
%
%
\definecolor{mycolor1}{rgb}{0.00000,0.44700,0.74100}%
\definecolor{mycolor2}{rgb}{255,0,0}%
\begin{tikzpicture}

\begin{axis}[%
width=1.5in,
height=1.5in,
at={(0.959in,0.568in)},
scale only axis,
xmin=0,
xmax=1,
xminorticks=true,
xlabel style={font=\color{white!15!black}},
xlabel={$\lambda$},
ymode=log,
ymin=9e-5,
ymax=1e-2,
yminorticks=true,
ylabel style={font=\color{white!15!black}},
ylabel={Error},
axis background/.style={fill=white},
legend style={at={(0.03,0.98)}, anchor=north west,legend cell align=left, font=\scriptsize, align=left, draw=white!15!black}
]
\addplot [color=mycolor1, line width=1.1pt, mark size=3.5pt, mark=o, mark options={solid, mycolor1}]
  table[row sep=crcr]{%
0.183939720585721	0.000281897647205952\\
0.226890261607525	0.000133002896836569\\
0.269840802629329	0.000128989375155873\\
0.312791343651134	0.000242653580766222\\
0.355741884672938	0.000378017622800176\\
0.398692425694742	0.000531454240670642\\
0.441642966716546	0.000684871491575686\\
0.48459350773835	0.000838269307485736\\
0.527544048760154	0.000991647620428483\\
0.570494589781958	0.00114500636249268\\
0.613445130803763	0.00129834546582794\\
0.656395671825567	0.00145166486264533\\
0.699346212847371	0.00160496448522356\\
0.742296753869175	0.00175824426590579\\
0.785247294890979	0.00191150413710284\\
0.828197835912783	0.00206474403129615\\
0.871148376934588	0.0022179638810378\\
0.914098917956392	0.0023711636189529\\
0.957049458978196	0.00252434317774059\\
1	0.00267750249017803\\
};
\addlegendentry{\textsc{sle} $\mu$-mode}

\addplot [color=mycolor2, line width=1.1pt, mark size=3.5pt, mark=x, mark options={solid, mycolor2}]
  table[row sep=crcr]{%
0.218	0.000309314170510721\\
0.226890261607525	0.00029016570275568\\
0.269840802629329	0.000203978432914627\\
0.312791343651133	0.000131746481806065\\
0.355741884672937	9.86060796118005e-05\\
0.398692425694742	0.000126258049467122\\
0.441642966716546	0.000151864843011902\\
0.48459350773835	0.000180430452670029\\
0.527544048760154	0.000200184656430159\\
0.570494589781959	0.000211132149076035\\
0.613445130803762	0.000213277789102244\\
0.656395671825566	0.000206626598380186\\
0.699346212847371	0.000191183760991773\\
0.742296753869175	0.000166954622140144\\
0.785247294890979	0.000133944688017173\\
0.828197835912783	0.000120776574307084\\
0.871148376934588	0.000165910170603296\\
0.914098917956392	0.000223677190716389\\
0.957049458978197	0.000289213824174373\\
1	0.000362504615692125\\
};
\addlegendentry{\textsc{le} $\mu$-mode}

\end{axis}
\end{tikzpicture}%
%
%
\definecolor{mycolor5}{rgb}{0.46600,0.67400,0.18800}%
\definecolor{mycolor7}{rgb}{0.63500,0.07800,0.18400}%
\definecolor{mycolor8}{rgb}{255,0,205}%
\begin{tikzpicture}

\begin{axis}[%
width=1.5in,
height=1.5in,
at={(0.959in,0.568in)},
scale only axis,
xmin=0,
xmax=1,
xminorticks=true,
xlabel style={font=\color{white!15!black}},
xlabel={$\lambda$},
ymode=log,
ymin=6e-8,
ymax=2e-5,
yminorticks=true,
ylabel style={font=\color{white!15!black}},
ylabel={Error},
axis background/.style={fill=white},
legend style={at={(0.03,0.98)}, anchor=north west,legend cell align=left, font=\scriptsize, align=left, draw=white!15!black}
]
\addplot [color=mycolor5, line width=1.1pt, mark size=3.5pt, mark=triangle, mark options={solid, mycolor5}]
  table[row sep=crcr]{%
0.183	4.14275960491764e-07\\
0.226	4.22279946637184e-07\\
0.269	4.69614378399604e-07\\
0.312	4.81457970515478e-07\\
0.355	4.89440027142871e-07\\
0.398	4.46936337306608e-07\\
0.441	3.53989843261546e-07\\
0.484	2.13711439724768e-07\\
0.527	1.99318575209476e-07\\
0.57	3.81525987944329e-07\\
0.613	6.16524544469315e-07\\
0.656	9.35044644427919e-07\\
0.699	1.29409866202292e-06\\
0.742	1.70543398121977e-06\\
0.785	2.20038097364852e-06\\
0.828	2.74538577915074e-06\\
0.871	3.34040578660628e-06\\
0.914	3.98539841622257e-06\\
0.957	4.6803211210312e-06\\
1	5.42513138688791e-06\\
};
\addlegendentry{\textsc{sl2} $\mu$-mode}

\addplot [color=mycolor7, line width=1.1pt, mark size=3.5pt, mark=triangle, mark options={solid, rotate=90, mycolor7}]
  table[row sep=crcr]{%
0.301	2.50601176038894e-07\\
0.323789473684211	2.43391326540239e-07\\
0.366052631578948	2.24347819964541e-07\\
0.408315789473684	1.99258425623005e-07\\
0.450578947368421	1.6968757950029e-07\\
0.492842105263158	1.37300610205054e-07\\
0.535105263157894	1.11721492677593e-07\\
0.577368421052631	8.33757739259383e-08\\
0.619631578947369	6.51782856916505e-08\\
0.661894736842105	8.5254407421719e-08\\
0.704157894736841	1.14374482534955e-07\\
0.746421052631579	1.54410934317242e-07\\
0.788684210526315	2.0836528782819e-07\\
0.830947368421052	2.76443224773086e-07\\
0.87321052631579	3.60683988819184e-07\\
0.915473684210525	4.6312192749934e-07\\
0.957736842105263	5.85786432026113e-07\\
1	7.30701844830458e-07\\
};
\addlegendentry{\textsc{l2a} $\mu$-mode}


\end{axis}
\end{tikzpicture}%
 \\
%
%
\definecolor{mycolor1}{rgb}{0.00000,0.44700,0.74100}%
\definecolor{mycolor2}{rgb}{255,0,0}%
\begin{tikzpicture}

\begin{axis}[%
width=1.5in,
height=1.5in,
at={(0.959in,0.568in)},
scale only axis,
xmin=0,
xmax=1,
xminorticks=true,
xlabel style={font=\color{white!15!black}},
xlabel={$\lambda$},
ymode=log,
ymin=9e-5,
ymax=1e-2,
yminorticks=true,
ylabel style={font=\color{white!15!black}},
ylabel={Error},
axis background/.style={fill=white},
legend style={at={(0.03,0.98)}, anchor=north west,legend cell align=left, font=\scriptsize, align=left, draw=white!15!black}
]
\addplot [color=mycolor1, line width=1.1pt, mark size=3.5pt, mark=o, mark options={solid, mycolor1}]
  table[row sep=crcr]{%
0.183939720585721	0.00483298488501304\\
0.226890261607525	0.00473509745568383\\
0.269840802629329	0.00463741913279213\\
0.312791343651134	0.00457692636330783\\
0.355741884672938	0.00452796362523569\\
0.398692425694742	0.004483658808636\\
0.441642966716546	0.00447938609218973\\
0.48459350773835	0.00447526002308322\\
0.527544048760154	0.00447446486702133\\
0.570494589781958	0.00451110862192536\\
0.613445130803763	0.00454785878349634\\
0.656395671825567	0.00458471498586015\\
0.699346212847371	0.00462167686418354\\
0.742296753869175	0.00465874405467301\\
0.785247294890979	0.00469934838555418\\
0.828197835912783	0.00493347522981767\\
0.871148376934588	0.00518245719661039\\
0.914098917956392	0.0054456190864959\\
0.957049458978196	0.00570894209808483\\
1	0.00597242593687977\\
};

\addplot [color=mycolor2, line width=1.1pt, mark size=3.5pt, mark=x, mark options={solid, mycolor2}]
  table[row sep=crcr]{%
0.218	0.000655604746165376\\
0.226890261607525	0.000648033788685393\\
0.269840802629329	0.000607614176838867\\
0.312791343651134	0.000561788884084264\\
0.355741884672938	0.00053107956647493\\
0.398692425694742	0.000497686001001886\\
0.441642966716546	0.000461611262024785\\
0.48459350773835	0.000422861686683215\\
0.527544048760154	0.000381508780000129\\
0.570494589781958	0.000338161612732082\\
0.613445130803763	0.000339892714680178\\
0.656395671825567	0.000459689428070053\\
0.699346212847371	0.000584513726900978\\
0.742296753869175	0.000714286266398498\\
0.785247294890979	0.000848923913607934\\
0.828197835912783	0.000988339974371481\\
0.871148376934588	0.00113244442558605\\
0.914098917956392	0.00128114415210834\\
0.957049458978196	0.00143434318785452\\
1	0.00159194296027945\\
};

\end{axis}
\end{tikzpicture}%
%
%
\definecolor{mycolor5}{rgb}{0.46600,0.67400,0.18800}%
\definecolor{mycolor7}{rgb}{0.63500,0.07800,0.18400}%
\definecolor{mycolor8}{rgb}{255,0,205}%
\begin{tikzpicture}

\begin{axis}[%
width=1.5in,
height=1.5in,
at={(0.959in,0.568in)},
scale only axis,
xmin=0,
xmax=1,
xminorticks=true,
xlabel style={font=\color{white!15!black}},
xlabel={$\lambda$},
ymode=log,
ymin=6e-8,
ymax=2e-5,
yminorticks=true,
ylabel style={font=\color{white!15!black}},
ylabel={Error},
axis background/.style={fill=white},
legend style={at={(0.03,0.98)}, anchor=north west,legend cell align=left, font=\footnotesize, align=left, draw=white!15!black}
]
\addplot [color=mycolor5, line width=1.1pt, mark size=3.5pt, mark=triangle, mark options={solid, mycolor5}]
  table[row sep=crcr]{%
0.183	6.07286571118913e-06\\
0.226	5.11888255211927e-06\\
0.269	4.35039853935415e-06\\
0.312	4.19163587764883e-06\\
0.355	3.99581497470838e-06\\
0.398	3.83633368120207e-06\\
0.441	3.56691849353103e-06\\
0.484	3.3093847713566e-06\\
0.527	3.18394835541132e-06\\
0.57	3.16196728821978e-06\\
0.613	3.40277077110451e-06\\
0.656	4.48171187077637e-06\\
0.699	5.64375130705773e-06\\
0.742	6.88239902336294e-06\\
0.785	8.19724969099228e-06\\
0.828	9.58790395542801e-06\\
0.871	1.10539683145599e-05\\
0.914	1.25950549982428e-05\\
0.957	1.42107818518016e-05\\
1	1.590077222225e-05\\
};

\addplot [color=mycolor7, line width=1.1pt, mark size=3.5pt, mark=triangle, mark options={solid, rotate=90, mycolor7}]
  table[row sep=crcr]{%
0.301	8.71479117870155e-07\\
0.323789473684211	7.64069072935429e-07\\
0.366052631578948	6.53456520209592e-07\\
0.408315789473684	5.64771672006623e-07\\
0.450578947368421	4.80928150253386e-07\\
0.492842105263158	4.35790914346495e-07\\
0.535105263157894	7.60796017706776e-07\\
0.577368421052631	1.29182591813901e-06\\
0.619631578947369	1.94728295841611e-06\\
0.661894736842105	2.73912253444158e-06\\
0.704157894736841	3.67917803062021e-06\\
0.746421052631579	4.77916005245137e-06\\
0.788684210526315	6.05065572869368e-06\\
0.830947368421052	7.50512776262864e-06\\
0.87321052631579	9.15391355440067e-06\\
0.915473684210525	1.10082240614528e-05\\
0.957736842105263	1.30791426706566e-05\\
1	1.53776239653938e-05\\
};


\end{axis}
\end{tikzpicture}%
  \caption{Solution of problem~\eqref{eq:adr3d} rewritten as
           equation~\eqref{eq:adr3dstab},
           with various first order schemes (left) and second order schemes
           (right), for different values of $\lambda$. The number of time steps
           is fixed to $m=2^8$ for each method, while the number of degrees of
           freedom for each spatial variable is $N=2^4$. The advection
           parameter is either $b=-0.01$ (top) or  $b=-1$ (bottom).}
  \label{fig:adr3d_diffl}
\end{figure}

\begin{table}[htb!]
  \centering
  \begin{tabular}{c||cc|cc|cc|cc}
    & \multicolumn{4}{c|}{$b=-0.01$} & \multicolumn{4}{c}{$b=-1$}\\
    & \textsc{le} & \textsc{sle} & \textsc{l2a} & \textsc{sl2}&
    \textsc{le} & \textsc{sle} & \textsc{l2a} & \textsc{sl2}\\
    \hline
    \rule{0pt}{12pt} value of $\lambda$    &  0.36 & 0.27 & 0.62 & 0.53 &
 0.60 & 0.50 & 0.50 & 0.57
  \end{tabular}
  \caption{Values of $\lambda$ for the different schemes
  that are employed to integrate problem~\eqref{eq:adr3d} rewritten as
  equation~\eqref{eq:adr3dstab} (see also Figure~\ref{fig:adr3d_diffl}).}
  \label{tab:adr3d_lambdaopt}
\end{table}

For the performance results, we set the number of uni-directional degrees of
freedom to $N=60$. As term of comparison, we consider here the same time
integrators but applied to the original equation~\eqref{eq:adr3d}
spatially discretized with second order centered finite differences.
In addition, in this formulation, we also present the results obtained
with the exponential Euler scheme, with {\color{black}the second order 
  exponential Runge--Kutta integrator \textsc{erk2p1}, with the exponential
  Rosenbrock--Euler method \textsc{erbe} and with the two IMEX 
schemes \textsc{bfe} and \textsc{imex2}}.
We perform the simulations with a number of time integration steps equal
to $2^{\ell_1}$, with $\ell_1=9,\ldots,12$, for the first order methods, while
we consider $2^{\ell_2}$, with $\ell_2=5,\ldots,8$, for the second order
schemes.
Again, the final simulation time is set to $T=1/4$.
The outcome of the experiments with advection parameter $b=-0.01$
is collected in a CPU diagram in Figure~\ref{fig:cpudiagadr3d}.
\begin{figure}[htb!]
  \centering
%
%
\definecolor{mycolor1}{rgb}{0.00000,0.44700,0.74100}%
\definecolor{mycolor2}{rgb}{255,0,0}%
\definecolor{mycolor3}{rgb}{0.92900,0.69400,0.12500}%
\definecolor{mycolor9}{RGB}{0,101,0}%

\begin{tikzpicture}

\begin{axis}[%
width=2.2in,
height=2.2in,
at={(0.959in,0.568in)},
scale only axis,
xmode=log,
xmin=0.05,
xmax=300,
xminorticks=true,
xlabel style={font=\color{white!15!black}},
xlabel={Wall-clock time},
ymode=log,
ymin=5e-6,
ymax=2e-3,
yminorticks=true,
ylabel style={font=\color{white!15!black}},
ylabel={Error},
axis background/.style={fill=white},
title style={font=\bfseries},
legend style={at={(0.02,0.98)},legend cell align=left, anchor=north west, font=\scriptsize, align=left, draw=white!15!black}
]
\addplot [color=mycolor1, dashed, line width=1.1pt, mark size=3.5pt, mark=o, mark options={solid, mycolor1}]
  table[row sep=crcr]{%
25.680157	0.00107882594647254\\
51.620715	0.000539212293060429\\
103.908177	0.000269555882383992\\
209.419618	0.000134765483488522\\
};
\addlegendentry{\textsc{sle}}

\addplot [color=mycolor2, dashed, line width=1.1pt, mark size=3.5pt, mark=x, mark options={solid, mycolor2}]
  table[row sep=crcr]{%
23.676131	8.36025235110282e-05\\
46.514127	4.18251645974182e-05\\
93.268766	2.09184785678974e-05\\
190.321012	1.04606305340334e-05\\
};
\addlegendentry{\textsc{le}}

\addplot [color=mycolor3, dashed, line width=1.1pt, mark size=3.5pt, mark=+, mark options={solid, mycolor3}]
  table[row sep=crcr]{%
25.59576	0.000318901204945739\\
51.50993	0.000159392645270882\\
103.580812	7.96819503724432e-05\\
207.290615	3.98374972442656e-05\\
};
\addlegendentry{\textsc{ee}}

\addplot [color=mycolor9, dashed, line width=1.1pt, mark size=3.5pt, mark=asterisk, mark options={solid, mycolor9}]
  table[row sep=crcr]{%
7.926463	0.00107712022909256\\
12.456199	0.00053880902258037\\
24.057095	0.000269458242150756\\
47.316428	0.000134742644838531\\
};
\addlegendentry{\textsc{bfe}}

\addplot [color=mycolor1, line width=1.1pt, mark size=3.5pt, mark=o, mark options={solid, mycolor1}]
  table[row sep=crcr]{%
1.212272	6.37364756220927e-05\\
2.630004	3.1841847933755e-05\\
4.929396	1.59143950657671e-05\\
9.77229899999999	7.95563088396029e-06\\
};
\addlegendentry{\textsc{sle} $\mu$-mode}

\addplot [color=mycolor2, line width=1.1pt, mark size=3.5pt, mark=x, mark options={solid, mycolor2}]
  table[row sep=crcr]{%
0.854913	5.06648788040468e-05\\
1.759183	2.53143317817367e-05\\
2.894102	1.26526978901009e-05\\
6.177308	6.32528882897452e-06\\
};
\addlegendentry{\textsc{le} $\mu$-mode}
\end{axis}
\end{tikzpicture}%
%
%
\definecolor{mycolor5}{rgb}{0.46600,0.67400,0.18800}%
\definecolor{mycolor6}{rgb}{0.30100,0.74500,0.93300}%
\definecolor{mycolor7}{rgb}{0.63500,0.07800,0.18400}%
\definecolor{mycolor8}{rgb}{255,0,205}%
\definecolor{mycolor10}{RGB}{255,120,0}%
\begin{tikzpicture}

\begin{axis}[%
width=2.2in,
height=2.2in,
at={(0.959in,0.568in)},
scale only axis,
xmode=log,
xmin=0.02,
xmax=200,
xminorticks=true,
xlabel style={font=\color{white!15!black}},
xlabel={Wall-clock time},
ymode=log,
ymin=1e-7,
ymax=7e-3,
yminorticks=true,
ylabel style={font=\color{white!15!black}},
ylabel={Error},
axis background/.style={fill=white},
title style={font=\bfseries},
legend style={at={(0.02,0.98)},legend cell align=left, anchor=north west, font=\scriptsize, align=left, draw=white!15!black, legend columns=1}
]
\addplot [color=mycolor5, dashed, line width=1.1pt, mark size=3.5pt, mark=triangle, mark options={solid, mycolor5}]
  table[row sep=crcr]{%
4.779929	0.000199242134051036\\
6.872611	5.00518668768234e-05\\
13.318199	1.25450596162728e-05\\
26.391141	3.1404621423507e-06\\
};
\addlegendentry{\textsc{sl2}}

\addplot [color=mycolor7, dashed, line width=1.1pt, mark size=3.5pt, mark=triangle, mark options={solid, rotate=90, mycolor7}]
  table[row sep=crcr]{%
5.39108800000001	1.23412072778272e-05\\
9.004942	3.11791546298527e-06\\
17.86827	7.83555480015955e-07\\
35.642441	1.96358830350465e-07\\
};
\addlegendentry{\textsc{l2a}}

\addplot [color=mycolor6, dashed, line width=1.1pt, mark size=3.5pt, mark=triangle, mark options={solid, rotate=270, mycolor6}]
  table[row sep=crcr]{%
3.663248	0.000110814104154194\\
6.401514	2.79146867972098e-05\\
12.596095	7.00603719611751e-06\\
25.158589	1.75504646766175e-06\\
};
\addlegendentry{\textsc{erk2p1}}
\addplot [color=teal, dashed, line width=1.1pt, mark size=3.5pt, mark=o, mark options={solid, teal}]
  table[row sep=crcr]{%
2.23658	3.56064295158679e-05\\
3.519931	8.81339412190728e-06\\
6.839474	2.19236128608225e-06\\
13.596233	5.46666027656426e-07\\
};
\addlegendentry{\textsc{erbe}}

\addplot [color=mycolor10, dashed, line width=1.1pt, mark size=3.5pt, mark=square, mark options={solid, mycolor10}]
  table[row sep=crcr]{%
3.814792	3.54625458730456e-05\\
5.078465	8.95316674151588e-06\\
7.42563699999999	2.25003210703449e-06\\
11.696275	5.64049892559223e-07\\
};
\addlegendentry{\textsc{imex2}}

\addplot [color=mycolor5, line width=1.1pt, mark size=3.5pt, mark=triangle, mark options={solid, mycolor5}]
  table[row sep=crcr]{%
0.240523	1.37247433792116e-05\\
0.436755	3.30758972197615e-06\\
0.821155	8.11405815993641e-07\\
1.64709	2.00961677046773e-07\\
};
\addlegendentry{\textsc{sl2} $\mu$-mode}

\addplot [color=mycolor7, line width=1.1pt, mark size=3.5pt, mark=triangle, mark options={solid, rotate=90, mycolor7}]
  table[row sep=crcr]{%
0.150253	9.62946157173456e-06\\
0.279444	3.20471448920401e-06\\
0.589413	9.31242481904633e-07\\
1.064283	2.50525901126287e-07\\
};
\addlegendentry{\textsc{l2a} $\mu$-mode}

\end{axis}
\end{tikzpicture}%
  \caption{Results for the simulations of problem~\eqref{eq:adr3d}
  (dashed lines), rewritten as equation~\eqref{eq:adr3dstab}
    (solid lines), with $b=-0.01$ and different integrators of first order
    schemes (left, number of time
    steps $m=2^{\ell_1}$ with $\ell_1=9,\ldots,12$) and second order schemes
    (right, number of time steps
  $m=2^{\ell_2}$ with $\ell_2=5,\ldots,8$).}
  \label{fig:cpudiagadr3d}
\end{figure}
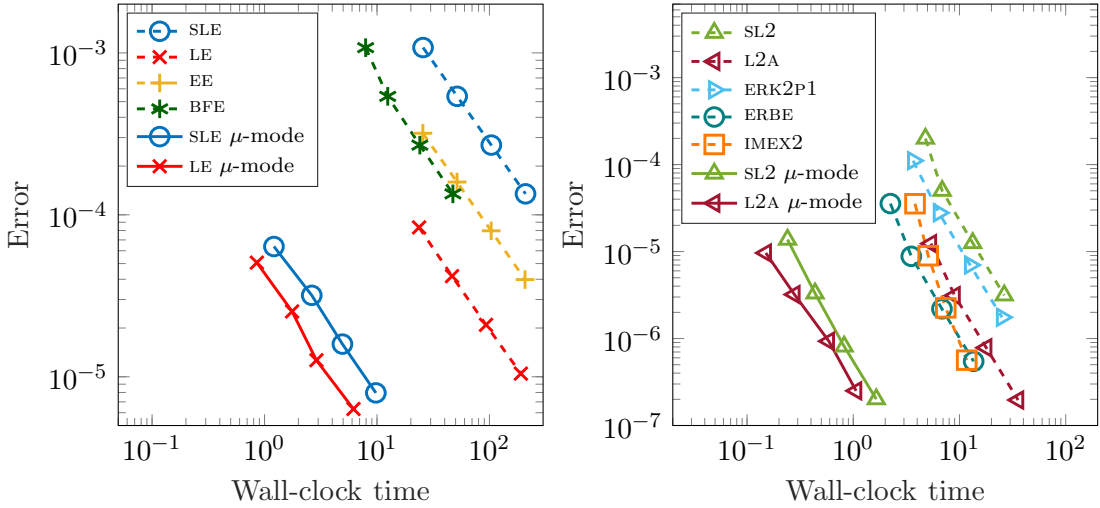

First of all, concerning the schemes of first order, we observe that in the
original formulation (i.e., the dashed lines in the plot)
the Lawson--Euler scheme is the one which performs best. Indeed,
comparing it with the other exponential methods, it is the one which reaches
the smallest error, with basically the same computational time.
The backward-forward Euler method performs slightly better in terms of
wall-clock
time, but the error is larger than both the exponential Euler method and
the Lawson--Euler scheme. Hence, overall, it is not the preferred method.
In any case, we observe that the proposed approach, i.e., solving instead
equation~\eqref{eq:adr3dstab} with $\mu$-mode techniques (solid lines) is
effective. Indeed, in this formulation both the Lawson--Euler and the stabilized
Lawson--Euler methods perform better than the corresponding counterparts,
both in terms of error and of wall-clock time, with a slight
advantage for the former.
Similar considerations can be drawn for the second order schemes. In fact,
using the proposed approach, the Lawson2a
scheme performs best, with a considerable advantage in performance
in the $\mu$-mode realization.

We then repeat the experiment by increasing the magnitude of advection parameter, i.e.,
setting it to $b=-1$. The results are collected in
Figure~\ref{fig:cpudiagadr3dadv}.
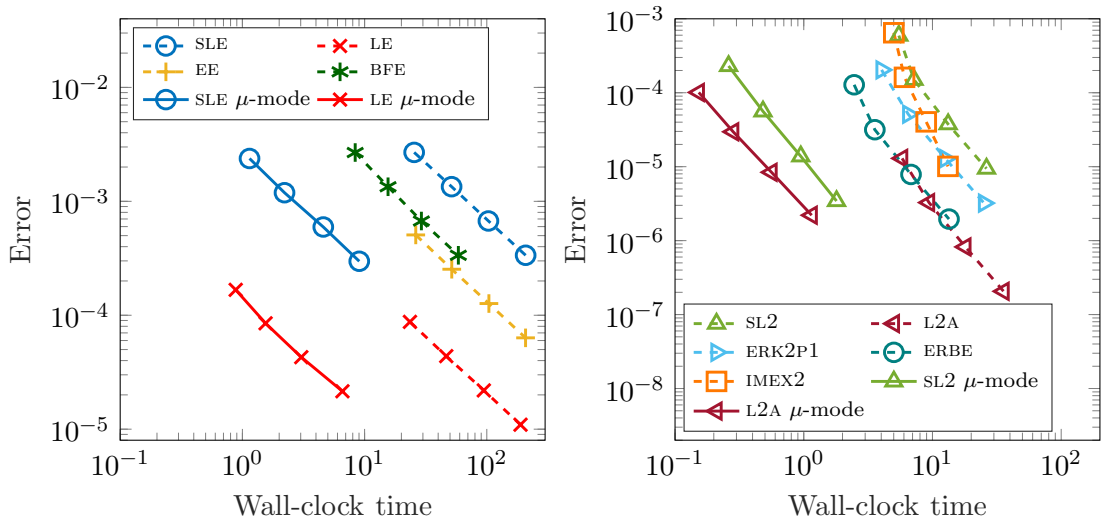
\begin{figure}[htb!]
  \centering
%
%
\definecolor{mycolor1}{rgb}{0.00000,0.44700,0.74100}%
\definecolor{mycolor2}{rgb}{255,0,0}%
\definecolor{mycolor3}{rgb}{0.92900,0.69400,0.12500}%
\definecolor{mycolor9}{RGB}{0,101,0}%
\begin{tikzpicture}

\begin{axis}[%
width=2.2in,
height=2.2in,
at={(0.959in,0.568in)},
scale only axis,
xmode=log,
xmin=0.1,
xmax=300,
xminorticks=true,
xlabel style={font=\color{white!15!black}},
xlabel={Wall-clock time},
ymode=log,
ymin=8e-6,
ymax=4e-2,
yminorticks=true,
ylabel style={font=\color{white!15!black}},
ylabel={Error},
axis background/.style={fill=white},
title style={font=\bfseries},
legend style={at={(0.03,0.98)}, anchor=north west,legend cell align=left, font=\scriptsize, align=left, draw=white!15!black, legend columns=2}
]
\addplot [color=mycolor1, dashed, line width=1.1pt, mark size=3.5pt, mark=o, mark options={solid, mycolor1}]
  table[row sep=crcr]{%
25.438148	0.00268820880068462\\
51.517292	0.00134442741127158\\
103.101919	0.000672286930716923\\
207.865247	0.000336160454507394\\
};
\addlegendentry{\textsc{sle}}

\addplot [color=mycolor2, dashed, line width=1.1pt, mark size=3.5pt, mark=x, mark options={solid, mycolor2}]
  table[row sep=crcr]{%
23.500029	8.74855741255474e-05\\
46.578932	4.37686074984193e-05\\
94.514226	2.18909369086836e-05\\
188.776346	1.09473011414144e-05\\
};
\addlegendentry{\textsc{le}}

\addplot [color=mycolor3, dashed, line width=1.1pt, mark size=3.5pt, mark=+, mark options={solid, mycolor3}]
  table[row sep=crcr]{%
26.217854	0.000507076254462362\\
51.66737	0.000253436274219673\\
104.03661	0.000126692254473531\\
207.545419	6.33392822251279e-05\\
};
\addlegendentry{\textsc{ee}}

\addplot [color=mycolor9, dashed, line width=1.1pt, mark size=3.5pt, mark=asterisk, mark options={solid, mycolor9}]
  table[row sep=crcr]{%
8.37649	0.00268173217211499\\
15.554716	0.00134278411108726\\
29.160018	0.000671872793198899\\
58.51391	0.000336056614389178\\
};
\addlegendentry{\textsc{bfe}}

\addplot [color=mycolor1, line width=1.1pt, mark size=3.5pt, mark=o, mark options={solid, mycolor1}]
  table[row sep=crcr]{%
1.143853	0.00237663440242041\\
2.211812	0.00118992720863273\\
4.591547	0.000595366794252353\\
9.042955	0.000297784405186162\\
};
\addlegendentry{\textsc{sle} $\mu$-mode}

\addplot [color=mycolor2, line width=1.1pt, mark size=3.5pt, mark=x, mark options={solid, mycolor2}]
  table[row sep=crcr]{%
0.881744	0.00016714841187133\\
1.547949	8.4970683175008e-05\\
3.033714	4.28468395634164e-05\\
6.587659	2.15126658317488e-05\\
};
\addlegendentry{\textsc{le} $\mu$-mode}

\end{axis}
\end{tikzpicture}%
%
%
\definecolor{mycolor5}{rgb}{0.46600,0.67400,0.18800}%
\definecolor{mycolor6}{rgb}{0.30100,0.74500,0.93300}%
\definecolor{mycolor7}{rgb}{0.63500,0.07800,0.18400}%
\definecolor{mycolor8}{rgb}{255,0,205}%
\definecolor{mycolor10}{RGB}{255,120,0}%
\begin{tikzpicture}

\begin{axis}[%
width=2.2in,
height=2.2in,
at={(0.959in,0.568in)},
scale only axis,
xmode=log,
xmin=0.1,
xmax=200,
xminorticks=true,
xlabel style={font=\color{white!15!black}},
xlabel={Wall-clock time},
ymode=log,
ymin=2e-09,
ymax=0.001,
yminorticks=true,
ylabel style={font=\color{white!15!black}},
ylabel={Error},
axis background/.style={fill=white},
title style={font=\bfseries},
legend style={at={(0.02,0.02)}, anchor=south west,legend cell align=left, font=\scriptsize, align=left, draw=white!15!black, legend columns=2}
]
\addplot [color=mycolor5, dashed, line width=1.1pt, mark size=3.5pt, mark=triangle, mark options={solid, mycolor5}]
  table[row sep=crcr]{%
5.479206	0.000592709698642253\\
7.189398	0.000149716438861676\\
13.234701	3.76042676665579e-05\\
26.290358	9.42214541359262e-06\\
};
\addlegendentry{\textsc{sl2}}

\addplot [color=mycolor7, dashed, line width=1.1pt, mark size=3.5pt, mark=triangle, mark options={solid, rotate=90, mycolor7}]
  table[row sep=crcr]{%
5.810147	1.29547754885674e-05\\
9.40021100000001	3.27043314795231e-06\\
18.018093	8.22005060875438e-07\\
35.677572	2.06402109147391e-07\\
};
\addlegendentry{\textsc{l2a}}


\addplot [color=mycolor6, dashed, line width=1.1pt, mark size=3.5pt, mark=triangle, mark options={solid, rotate=270, mycolor6}]
  table[row sep=crcr]{%
4.023071	0.000203728635110492\\
6.48122299999999	5.12422236701668e-05\\
12.602713	1.2849458584198e-05\\
25.148889	3.21653923971702e-06\\
};
\addlegendentry{\textsc{erk2p1}}

\addplot [color=teal, dashed, line width=1.1pt, mark size=3.5pt, mark=o, mark options={solid, teal}]
  table[row sep=crcr]{%
2.463399	0.00012812809672794\\
3.566606	3.16198377051595e-05\\
6.812385	7.85318915325016e-06\\
13.439096	1.95741663318208e-06\\
};
\addlegendentry{\textsc{erbe}}

\addplot [color=mycolor10, dashed, line width=1.1pt, mark size=3.5pt, mark=square, mark options={solid, mycolor10}]
  table[row sep=crcr]{%
5.009663	0.000647497687360482\\
6.06461	0.000161632319547973\\
8.94076199999999	4.03864301541777e-05\\
13.1919	1.00939641526294e-05\\
};
\addlegendentry{\textsc{imex2}}

\addplot [color=mycolor5, line width=1.1pt, mark size=3.5pt, mark=triangle, mark options={solid, mycolor5}]
  table[row sep=crcr]{%
0.260655	0.000230076903649126\\
0.482324	5.60027622282653e-05\\
0.948382	1.3822232329729e-05\\
1.797883	3.43456773072877e-06\\
};
\addlegendentry{\textsc{sl2} $\mu$-mode}

\addplot [color=mycolor7, line width=1.1pt, mark size=3.5pt, mark=triangle, mark options={solid, rotate=90, mycolor7}]
  table[row sep=crcr]{%
0.152925	0.000100984715414219\\
0.282913	2.96470083154809e-05\\
0.561737	8.40036139833613e-06\\
1.152559	2.20578505394479e-06\\
};
\addlegendentry{\textsc{l2a} $\mu$-mode}

\end{axis}

\end{tikzpicture}%
  \caption{Results for the simulations of problem~\eqref{eq:adr3d}
  (dashed lines), eventually rewritten as equation~\eqref{eq:adr3dstab}
    (solid lines), with $b=-1$ and different integrators  of first order
    (left, number of time
  steps $m=2^{\ell_1}$ with $\ell_1=9,\ldots,12$) and
  second order (right, number of time steps
  $m=2^{\ell_2}$ with $\ell_2=5,\ldots,8$).}
  \label{fig:cpudiagadr3dadv}
\end{figure}
In this case, we first of all note that the stabilized
schemes applied to equation~\eqref{eq:adr3dstab} show errors which are either
comparable (stabilized Lawson--Euler) or
even lower (stabilized Lawson2) with respect to their counterparts
in the original formulation. This is not true for the Lawson--Euler and the
Lawson2a schemes, which in fact increase the errors when employing the proposed
approach.
Nevertheless, the gain in computational time is so considerable that
overall the $\mu$-mode implementation is the preferred one.
In particular, the methods which perform best are the Lawson--Euler scheme
and the Lawson2a method for first and second order, respectively.

\section{Conclusion}\label{sec:concstab}
{\color{black}In this manuscript, we presented an effective acceleration 
approach to numerically solve semilinear advection-diffusion-reaction equations
in the framework of exponential integrators.
The technique is based on a reformulation of the original equation, driven by a
linear stability analysis of a simpler model, and
allows for the employment of ad-hoc efficient methods for the time integration
(FFT or $\mu$-mode based, for instance).
In this context, we also presented two new schemes of Lawson type which have 
improved unconditional stability bounds compared to methods already available
in the literature, and which appear to perform comparably with respect to well-known
integrators in most of the situations.

We conducted numerical examples and performance comparisons on several
semilinear ad\-vection-diffusion-reaction equations in one, two, and 
three space dimensions. The outcomes show the superiority of 
the proposed technique for exponential integrators, always outperforming
straightforward implementations (up to a factor of roughly 15 in the 
three dimensional case) 
and often obtaining better results with respect to popular IMEX schemes.

As further developments, we plan to deeply investigate alternative
choices for the approximation operator $A$ (making it vary
according to the time evolution of the solution by relating it to
the Jacobian of the equation, for instance) as well as to consider nonlinear
diffusion terms, both from a
theoretical and a practical point of view. Also, we intend
to apply the proposed approach to different classes of PDEs arising from science
and engineering problems.}

\section*{Acknowledgments}
M.~C.~and F.~C.~acknowledge partial support from the Program Ricerca di Base
2019 No.~RBVR199YFL
of the University of Verona entitled ``Geometric Evolution of Multi Agent
Systems''.
L.~E.~acknowledges support from the Austrian Science Fund (FWF) --- project id:
P32143-N32.

\section*{Appendix}
  For the convenience of the reader, we list here the integrators that
  have been mentioned and studied in the paper.
  We suppose that the equation under study is given in the following
  abstract form
  \begin{equation*}
    u'(t) = Au(t) + g(t,u(t)) = F(t,u(t)),
  \end{equation*}
  where $A$ is a generic linear operator and $g$ is a nonlinear function.

  \bigskip
  \noindent\textbf{Lawson type exponential integrators}

  \bigskip
  \noindent The \emph{Lawson--Euler} scheme is given by
  \begin{equation*}
    u^{n+1}=\ee^{\tau A}(u^n + \tau g(t_n,u^n)).
  \end{equation*}
  The \emph{stabilized Lawson--Euler} scheme is given by
  \begin{equation*}
    u^{n+1}=u^n+\tau\ee^{\tau A}F(t_n,u^n).
  \end{equation*}
  The \emph{Lawson2a} scheme is given by
  \begin{equation*}
      U = \ee^{\frac{\tau}{2}A}\left(u^n +
           \frac{\tau}{2}g(t_n,u^n)\right), \quad
      u^{n+1} = \ee^{\tau A}u^n +
                 \tau\ee^{\frac{\tau}{2}A}g\left(t_n+\frac{\tau}{2},U\right).
  \end{equation*}
  The \emph{Lawson2b} scheme is given by
  \begin{equation*}
      U = \ee^{\tau A}\left(u^n + \tau g(t_n,u^n)\right), \quad
      u^{n+1} = \ee^{\tau A}\left(u^n + \frac{\tau}{2}g(t_n,u^n)\right)
      + \frac{\tau}{2}g(t_n+\tau,U).
  \end{equation*}
  The \emph{stabilized Lawson2} scheme is given by
  \begin{equation*}
      U = u^n + \alpha\tau\ee^{\alpha\tau A}F(t_n,u^n), \quad
      u^{n+1} = u^n + \tau\ee^{\frac{\tau}{2} A}F(t_n,u^n) +
                 \frac{\tau}{2\alpha}\ee^{\tau A}
                 (g(t_n+\alpha\tau,U)-g(t_n,u^n)).
  \end{equation*}

  \bigskip
  \noindent\textbf{Classical exponential integrators}

  \bigskip
  \noindent The \emph{exponential Euler} scheme is given by
  \begin{equation*}
    u^{n+1} = u^n + \tau \varphi_1(\tau A)F(t_n,u^n).
  \end{equation*}
  The \emph{exponential Runge--Kutta} scheme of second order involving
  the $\varphi_2$ function is given by
  \begin{equation*}
    \begin{aligned}
      U &= u^n + c_2\tau \varphi_1(c_2\tau A)F(t_n,u^n), \\
      u^{n+1} &= u^n + \tau \varphi_1(\tau A)F(t_n,u^n) +
                 \frac{\tau}{c_2}\varphi_2(\tau A)(g(t_n+c_2\tau,U)-g(t_n,u^n)).
    \end{aligned}
  \end{equation*}
  The \emph{exponential Runge--Kutta} scheme of second order involving the
  $\varphi_1$ function is given by
  \begin{equation*}
    \begin{aligned}
      U &= u^n + c_2\tau \varphi_1(c_2\tau A)F(t_n,u^n), \\
      u^{n+1} &= u^n + \tau \varphi_1(\tau A)F(t_n,u^n) +
                 \frac{\tau}{2c_2}\varphi_1(\tau A)(g(t_n+c_2\tau,U)-g(t_n,u^n)).
    \end{aligned}
  \end{equation*}
      {\color{black}The \emph{exponential Rosenbrock--Euler} scheme
        for autonomous problems with $F(u(t))=Au(t)+g(u(t))$ is given by
  \begin{equation*}
    u^{n+1}=u^n+\tau\varphi_1(\tau J_n)F(u^n), \quad J_n = A+\frac{\partial g}{\partial u}(u^n) .
  \end{equation*}
  }

  \bigskip
  \noindent\textbf{IMEX schemes}

  \bigskip
  \noindent The \emph{backward-forward Euler} scheme is given by
  \begin{equation*}
    (I-\tau A )u^{n+1} = u^n + \tau g(t_n,u^n).
  \end{equation*}
  The second order \emph{implicit-explicit2} scheme proposed
  in reference~\cite{WZWS20} is given by
  \begin{equation*}
      \left(I-\frac{\tau}{2} A\right)U = u^n + \frac{\tau}{2} g(t_n,u^n), \quad
      \left(I-\frac{\tau}{2} A\right)u^{n+1} = u^n + \frac{\tau}{2}A u^n +\tau g\left(t_n+\frac{\tau}{2},U\right).
  \end{equation*}

\bibliographystyle{plain}
\bibliography{bibliostab}
\end{document}